\documentclass[a4paper,10pt]{amsart}
\usepackage{amssymb,amsmath,amsfonts,amsthm}
\usepackage{graphicx}
\graphicspath{ {./images/} }
 \usepackage{ mathrsfs }
 
\usepackage{psfrag}
\usepackage{mathtools}
\usepackage{color}
\usepackage{enumitem}

\theoremstyle{plain}
\newtheorem{main}{Theorem}

\newtheorem{maincor}[main]{Corollary}
\newtheorem{theorem}{Theorem}[section]
\newtheorem{lemma}[theorem]{Lemma}
\newtheorem{proposition}[theorem]{Proposition}

\theoremstyle{remark}
\newtheorem{remark}[theorem]{Remark}
\newtheorem{definition}{Definition}

\newtheorem{conjecture}[theorem]{Conjecture}

\newcommand{\C}{\operatorname{C}}
\newcommand{\Sing}{\operatorname{Sing}}

\newcommand{\diam}{\operatorname{diam}}

           \def\ea{\end{array}}
          \def\ec{\end{center}}
     \def\ed{\end{description}}
        \def\ee{\end{equation}}
       \def\eea{\end{eqnarray}}
     \def\eeaa{\end{eqnarray*}}
 \def\et{\end{thebibliography}}

\def\Orb{{\rm Orb}}

\def\Cl{{\rm Cl}}

\def\Sing{{\rm Sing}}
\def\Reg{{\rm Reg}}
\def\Ind{{\rm Ind}}

\def\supp{\operatorname{supp}}

\def\cL{{\mathcal L}}

\def\cU{{\mathcal U}}

\def\cR{{\mathcal R}}

\def\cF{{\mathcal F}}

\def\cN{{\mathcal N}}

\def\cR{{\mathcal R}}

\def\RR{{\mathbb R}}

\def\NN{{\mathbb N}}

\title[Singular star flows]{An entropy dichotomy for singular star flows}

\author{Maria Jos\' e Pacifico, Fan Yang and Jiagang Yang}
\date{\today}
\thanks{M.J.P. and J.Y. are partially supported by CNPq, FAPERJ and  PROEX-CAPES. J.Y. is partially supported by NSFC 11871487 of China. F.Y. would like to thank the hospitality of Southern University of Science and Technology of China (SUSTC), where part of this work is done.}

\address{Instituto de Matem\'atica, Universidade Federal do Rio de Janeiro, C. P. 68.530, CEP 21.945-970,  Rio de Janeiro, RJ, Brazil.}
 \email{pacifico@im.ufrj.br }
\address{Department of Mathematics, Michigan State University, East Lansing, MI, USA.}
\email{yangfa31@msu.edu}
\address{Department of Mathematics, Southern University of Science and Technology of China, Guangdong, China; and 
Departamento de Geometria, Instituto de Matem\'atica e Estat\'istica, Universidade
Federal Fluminense, Niter\'oi, Brazil.}
\email{yangjg\@@impa.br}
\begin{document}

\begin{abstract}We show that  non-trivial chain recurrent classes for  generic $C^1$ star flows satisfy a dichotomy: either they have zero topological entropy, or they must be isolated. Moreover, chain recurrent classes for generic star flows with zero entropy must be sectional hyperbolic, and cannot be detected by any non-trivial ergodic invariant probability. As a result, we show that $C^1$ generic star flows have only finitely many Lyapunov stable chain recurrent classes. 
\end{abstract}

\maketitle

\section{Introduction}
A milestone towards the final proof of the stability conjecture for diffeomorphisms is the independent discovery of the star systems by Liao~\cite{Liao81} and Ma\~n\'e~\cite{Mane}. Recall that a diffeomorphism with star condition is a $C^1$ diffeomorphism for which all the periodic orbits of all nearby diffeomorphisms are hyperbolic (orbit-wise, not uniformly); in other words, there is no periodic orbit bifurcation. It was first proven in~\cite{Franks,Liao81} that $\Omega$-stability implies the star condition, and later proven by Hayashi~\cite{H92} with the help of the connecting lemma that star condition is equivalent to Axiom A with the no-cycle condition. For smooth vector fields without singularities, the same result was obtained by Gan and Wen~\cite{GW} using a different approach.

For singular flows, that is, flows exhibiting singularities, the situation is much more complicated. Here, let us introduce the precise definition of the star condition for singular flows, taking into account the presence of the singularities.

\begin{definition}
	A $C^1$ vector field $X$ is a {\em star vector field} if there exists a neighborhood $X\in \cU\subset \mathscr{X}^1(M)$ such that for all vector fields $Y\in \cU$,  all the  critical elements of $Y$ i.e., singularities and periodic orbits, are hyperbolic. The collection of $C^1$ star vector fields is denoted by $\mathscr{X}^1_*(M)$. 
\end{definition}

It is well known that singular star flows can exhibit rich dynamics and bifurcation phenomenon comparing to their non-singular counterpart. Certain singular star flows, such as the famous Lorenz flows, do not have  hyperbolic nonwandering sets~\cite{G76} and are not structural stable~\cite{GW79}. There are also examples where the set of the periodic orbits is not dense in the nonwandering set~\cite{Ding}, and where the nonwandering set is Axiom A but no-cycle condition fails~\cite{LW}.

In order to properly describe the hyperbolicity of an invariant set that contains singularity, Morales, Pacifico and Pujals~\cite{MPP99} proposed the notion of singular hyperbolicity for three-dimensional flows. They require the flow to have a one-dimensional uniformly contracting direction, and a two-dimensional subbundle
containing the flow direction of regular points, on which the 
flow is volume expanding. This is later generalized to higher dimensions as {\em sectional hyperbolicity}~\cite{LGW, MM}. See Definition~\ref{d.sectionalhyperbolic} in the next section for more details. 

Next, let us introduce the notion of chain recurrent classes, which plays an important role in the study of the stability conjecture. Roughly speaking, chain recurrent classes are the largest non-trivial invariant set of a topological dynamical system, outside which the system behaves like a gradient system.

\begin{definition}
	For $\varepsilon>0, T>0$, a finite sequence $\{x_i\}_{i=0}^n$ is called an {\em $(\varepsilon, T)$-chain} if there exists $\{t_i\}_{i=0}^{n-1}$ such that $t_i>T$ and $d(\phi_{t_i}(x_i),x_{i+1})<\varepsilon$  for all $i=0,\ldots, n-1$. We say that $y$ is {\em chain attainable from $x$}, if there exists $T>0$ such that for all $\varepsilon>0$, there exists an $(\varepsilon, T)$-chain $\{x_i\}_{i=0}^n$ with $x_0=x$ and $x_n=y$. It is straightforward to check that chain attainability is an equivalent relation on the closure of the set ${\{x: \forall t>0, \mbox{ there is an } (\varepsilon, T)\mbox{-chain with } x_0 = x_n = x\mbox{ and }\sum_i t_i > t\}}$. Each equivalent class under this relation is then called a {\em chain recurrent class}. A chain recurrent class $C$ is  {\em non-trivial} if it is not a singularity nor a periodic orbit.
\end{definition}

As the first step towards the complete understanding of singular star flows, the following conjecture was raised  in~\cite{ZGW}.
\begin{conjecture}\label{c.1}
	(Generic) singular star flows have only finitely many chain recurrent classes, all of which are sectional hyperbolic for $X$ or $-X$.
\end{conjecture}
A partial affirmative answer to this conjecture was obtained in~\cite{GSW}, where it is proven that for generic star flows, every non-trivial {\em Lyapunov stable} chain recurrent class is sectional hyperbolic. In fact, they obtained a complete characterization on the stable indices of singularities in the same chain recurrent classes. This result is crucial to our current paper, and is explained in more detail in Section~\ref{s.starflow}.

A breakthrough was later obtained by da Luz and Bonatti~\cite{BD, daLuz} which gives a negative answer to the second half of this conjecture. They construct an example which has two singularities with different indices robustly contained in the same chain recurrent class. As a result, such classes cannot be sectional hyperbolic. In~\cite{BD} they propose the notion of multi-singular hyperbolicity  and prove that, generically, star condition is equivalent to multi-singular hyperbolicity~\cite[Theorem 3]{BD}. To avoid technical difficulty, we will not provide the definition of multi-singular hyperbolicity here, and invite the interested readers to~\cite{CDYZ} for a simpler yet useful definition.

However, the first half of Conjecture~\ref{c.1} is still left unanswered: {\em must a (generic) singular star flow have only finitely many chain recurrent classes?} The goal of the current paper is to provide a partial answer to this question, using the recent progress in the entropy theory~\cite{PYY}:

\begin{main}\label{m.star}
There is a residual set $\cR$ of $C^1$ star flows, such that for every $X \in \cR$ and every non-trivial chain recurrent class $C$ of $X$, we have 
\begin{enumerate}
\item if $h_{top}(X|_C)>0$, then $C$ contains some periodic point $p$ and is isolated;
\item if $h_{top}(X|_C)=0$, then $C$ is sectional hyperbolic for $X$ or $-X$, and contains no periodic orbits. In this case, every ergodic invariant measure $\mu$ with $\supp(\mu)\subset C$ must satisfy $\mu = \delta_{\sigma}$ for some $\sigma\in \Sing(X)\cap C$.
\end{enumerate}
\end{main}
In other words, a chain recurrent class with zero topological entropy cannot be detected by any non-trivial invariant ergodic probability measure, in the sense that it can only support point masses of  singularities. 
We call such chain recurrent classes {\em singular aperiodic classes}.

Note that in the second case, $C$ cannot be Lyapunov stable due to \cite[Corollary D]{PYY}. Also note that the example constructed by Bonatti and da Luz belongs to the first case and is isolated.

An immediate corollary of Theorem~\ref{m.star} is:
\begin{maincor}\label{mc.lyapunovstable}
$C^1$ generic star flows have only finitely many Lyapunov stable chain recurrent classes.
\end{maincor}



Now the first half of Conjecture~\ref{c.1}, namely the finiteness of chain recurrent classes for singular star flows, is reduced to the following conjecture:
\begin{conjecture}
For (generic) singular star flows, singular aperiodic classes do not exist. Consequently, (generic) singular star flows have only finitely many chain recurrent classes, all of which are homoclinic classes of some periodic orbits.
\end{conjecture}

\subsubsection*{Organization of the paper} In Section~\ref{s.2} we provide the readers with preliminaries on singular flows: dominated splitting, (extended and scaled) linear Poincar\'e flows and Liao's shadowing lemma. We also include in Section~\ref{s.starflow} a classification on the singularities in the same chain recurrent classes, which is taken from~\cite{GSW}. Section~\ref{ss.singularity} contains a detailed analysis on the dynamics near Lorenz-like singularities and, more importantly, on the transverse intersection between invariant manifolds of nearby periodic orbits with that of points in the class. Finally, we provide the proof of Theorem~\ref{m.star} in Section~\ref{ss.starproof} by showing that every chain recurrent class with positive topological entropy must contain a periodic orbit and, consequently, contains all periodic orbits that are sufficiently close to the class. 

\section{Preliminaries}\label{s.2}
Throughout this paper, $X$ will be a $C^1$ vector field on a $d$-dimensional compact manifold $M$. Denote by $\Sing(X)$ (sometimes we also write $\Sing(\phi_t)$) the set of singularities of $X$, $\phi_t$ the flow generated by $X$, and $f = \phi_1$ the time-one map of $\phi_t$. We will write $\Phi_t$ for  the tangent flow, i.e., $\Phi_t  = D\phi_t: TM \to TM$.

This section includes the necessary background for the proof of Theorem~\ref{m.star}. Most notably, we will introduce the scaled and extended linear Poincar\'e flow by Liao~\cite{Liao, Liao96}, and the shadowing lemma which was first introduced by Liao~\cite{Liao, Liao96}, and further developed by Gan~\cite{G02}. We also collect some previously established results on generic singular star flows in Section~\ref{s.starflow}, most of which can be found in~\cite{GSW} and~\cite{LGW}.

\subsection{Dominated splitting and invariant cones}\label{s.2.1}

A dominated splitting for a flow $\phi_t$ is defined similarly to the case of diffeomorphisms. 
The invariant  set $\Lambda$ admits a {\em dominated splitting} $T_\Lambda M=E \oplus F$ if this splitting is
invariant under $\Phi_t$, and if there exist $C > 0$ and $\lambda < 1$ such that for every $x \in \Lambda$, and
every pair of unit vectors $u \in E_x$ and $v \in F_x$, one has
$$
\|(\Phi_t)_x(u)\| \le C\lambda^t\|(\Phi_t)_x(v)\| \mbox{ for } t > 0.
$$
We invite the readers to~\cite[Appendix B]{BDV} and~\cite{ArPa10} for more properties on the dominated splitting. The next lemma states the relation between dominated splitting for the flow and its time-one map.

\begin{lemma}\cite[Lemma 2.6]{PYY}
Let $\Lambda$ be an invariant set. A splitting $T_\Lambda M = E\oplus F$ is a dominated splitting for the flow $\phi_t|_\Lambda$ if and only if it
is a dominated splitting for the time-one map $f|_\Lambda$. Moreover, if $\phi_t|_\Lambda$ is transitive,
then we have either $X|_{\Lambda\setminus\Sing(X)} \subset E $ or $X|_{\Lambda\setminus\Sing(X)} \subset F.$
\end{lemma}

\begin{definition}\label{d.sectionalhyperbolic}
A compact invariant set $\Lambda$ of a flow $X$ is called {\em sectional hyperbolic}, if it admits a dominated splitting $E^s\oplus F^{cu}$, such that $E^s$ is uniformly contracting, and $F^{cu}$ is {\em sectional-expanding:}  there are constants $C,\lambda>0$ such that for every $x\in\Lambda$ and any subspace $V_x\subset F^{cu}_x$ with $\dim V_x \ge 2$, we have 
$$
|\det D\phi_t(x)|_{V_x}| \ge C e^{\lambda t} \mbox{ for all } t>0.
$$
We  call $\lambda$ the {\em sectional volume expanding rate on $F^{cu}$}. 
\end{definition}

\begin{remark}
If the dominated splitting $E\oplus F$ is sectional hyperbolic, then $\Phi_t$ on $E$ is uniformly contracting by definition. Since the flow speed $\|X(x)\|$ is bounded and thus cannot be backward exponentially expanding, we must have $X|_{\Lambda\setminus\Sing(X)} \subset F$. For more detail, see \cite[Lemma 3.10]{PYY}.
\end{remark}

Let $E\oplus F$ be a dominated splitting for the flow $\phi_t$. For $a > 0$ and $x \in M$, a {\em $(a, F)$-cone} on the tangent space $T_xM$ is defined as
$$
C_a(F_x) = \{v:  v = v_E + v_F \mbox{ where } v_E \in E, v_F\in F \mbox{ and } \|v_E\|< a\|v_F\|\}\cup\{0\}.
$$
When $a$ is sufficiently small, the cone field $C_a(F_x)$, $x \in M$, is forward invariant by $\Phi_1$,
i.e., there is $\lambda < 1$ such that for any $x \in  M$, $\Phi_1(C_a(F_x)) \subset  C_{\lambda a}(F_{f(x)})$. Similarly, we can define the $(a, E)$-cone $C_a(E_x)$, which is backward invariant by $\Phi_1$. When no
confusing is caused, we call the two families of cones by $F$ cones and $E$ cones.

For a $C^1$ disk with dimension at most $\dim F$, we say that $D$ is  {\em tangent to the $(a, F)$-cone} if for any $x \in D$, $T_xD \subset  C_a(F_x)$. The same can be said for the $(a,E)$-cone if $T_xD \subset  C_a(E_x)$ for every $x\in D$.

\subsection{Scaled linear Poincar\'e flows and a shadowing lemma by Liao}\label{ss.pesin}
In this section, $\mu$ will be a non-trivial 
ergodic measure with a dominated splitting $E\oplus F$ for the time-one map $f=\phi_1$ on  $\supp\mu$. 

The {\em linear Poincar\'e flow} $\psi_t$ is defined as following: denote the normal bundle
of $\phi_t$ over $\Lambda$ by 
$$
N_\Lambda = \bigcup_{x\in\Lambda\setminus\Sing(X)}N_x,
$$
where $N_x$ is the orthogonal complement of the flow direction $X(x)$, i.e.,
$$
N_x = \{v \in T_xM: v \perp X(x)\}.
$$
Denote the orthogonal projection of $T_xM$ to $N_x$ by $\pi_x$ and the projection of $T_\Lambda M$ to $N_\Lambda$ by $\pi$. Given $v \in N_x$ for a regular point $x \in
M \setminus  \Sing(X)$ and recalling that $\Phi_t$ is the tangent flow,  we can define $\psi_t(v)$ as the  orthogonal projection of $\Phi_t(v)$ onto $N_{\phi_t(x)}$, i.e.,
$$
\psi_t(v) = \pi_{\phi_t(x)}(\Phi_t(v)) = \Phi_t(v) -\frac{< \Phi_t(v), X(\phi_t(x)) >}{\|X(\phi_t(x))\|^2}X(\phi_t(x)),
$$
where $< .,. >$ is the inner product on $T_xM$ given by the Riemannian metric. The following is the flow version of the Oseledets theorem:
\begin{proposition}
For $\mu$ almost every $x$, there exists $k = k(x) \in \NN$ and real numbers
$$
\hat{\lambda}_1(x) > \cdots > \hat{\lambda}_k(x)
$$
and a $\psi_t$ invariant measurable splitting on the normal bundle:
$$
N_x = \hat{E}^1_x \oplus \cdots\oplus \hat{E}^k_x,
$$
such that
$$
\lim_{t\to\pm\infty}\frac1t\log \|\psi_t(v_i)\| = \hat{\lambda}_i(x)\mbox{ for every non-zero vector }v_i \in \hat{E}^i_x.
$$
\end{proposition}

Now we state the relation between Lyapunov exponents and the Oseledets splitting
for $\psi_t$ and for $f=\phi_1$:
\begin{theorem}
For $\mu$ almost every $x$, denote by $\lambda_1(x) > \cdots > \lambda_k(x)$ the Lyapunov
exponents and  
$$
T_xM = E^1_x \oplus\cdots\oplus E^k_x 
$$
the Oseledets splitting of $\mu$ for $f$. Then
$$
N_x = \pi_x(E^1_x)\oplus\cdots\oplus \pi_x(E^k_x)
$$ 
is the Oseledets splitting of $\mu$ for the linear Poincar\'e flow $\psi_t$. Moreover, the Lyapunov
exponents of $\mu$ (counting multiplicity) for  $\psi_t$ is the subset of the exponents
for $f$ obtained by removing one of the zero exponent which comes from the flow direction.
\end{theorem}

\begin{definition}
A non-trivial measure $\mu$ is called a {\em hyperbolic measure} for the flow $\phi_t$ if it is an ergodic
measure of $\phi_t$ and all the Lyapunov exponents for the linear Poincar\'e flow $\psi_t$ are non-vanishing. In other words, if we view $\mu$ as an invariant measure for the time-one map $f$, then $\mu$ has exactly one exponent which is zero, given by the flow direction.  We call the number of the
negative exponents of $\mu$, counting multiplicity, its \emph{index}.
\end{definition}

The {\em scaled linear Poincar\'e flow}, which we denote by $\psi^*_t$, is the normalization of $\psi_t$ using the flow speed:
\begin{equation}
\psi^*_t(v) = \frac{\|X(x)\|}{\|X(\phi_t(x))\|}\psi_t(v) = \frac{\psi_t(v)}{\|\Phi_t|_{<X(x)>}\|}.
\end{equation}  

We collect the following elementary properties of $\psi_t^*$.

The next lemma is first observed by Liao~\cite{Liao}, see also~\cite{GY} and~\cite[Section 4]{PYY} for the proof. 
\begin{lemma}\label{l.scaledflow}
$\psi^*_t$ is a bounded cocycle over $N_\Lambda$ in the following sense: for any $\tau>0$, there is $C_\tau>0$ such that for any $t\in[-\tau,\tau]$,
$$
\|\psi^*_t\|\le C_\tau.
$$
Furthermore, for every non-trivial ergodic measure $\mu$, the cocycles $\psi_t$ and $\psi^*_t$ have the same Lyapunov exponents and Oseledets splitting.
\end{lemma}


Next we describe the (quasi-)hyperbolicity for the scaled linear Poincar\'e flow $\psi^*_t$.

\begin{definition}\label{d12}
	For $T_0>0$, $\lambda\in (0,1)$, an orbit segment $\{\phi_t(x)\}_{[0,T]}$ is called $(\lambda,T_0)$-forward contracting for the bundle $E\subset N_x$, if there exists a partition
	$$
	0=t_0<t_1<\cdots<t_n=T, \mbox{\hspace{.2cm} where } t_{i+1}-t_i\in [T_0,2T_0], 
	$$
	such that for all $k=1,\ldots,n-1$,
	\begin{equation}\label{e.hyptime}
	\prod_{i=0}^{k-1}\|\psi^*_{t_{i+1} - t_i}|_{\psi_{t_i}(E)}\| \le \lambda^k.
	\end{equation}
	Similarly, an orbit segment $\{\phi_t(x)\}_{[0,T]}$ is called $(\lambda,T_0)$-backward contracting for the bundle $E\subset N_x$, if the orbit segment $\{\phi_{-t}(\phi_T(x))\}_{[0,T]}$ is $(\lambda,T_0)$-forward contracting for the flow $-X$. 
\end{definition}

\begin{definition}\label{d11}
For $T_0>0, \lambda\in(0,1)$, the orbit segment $\{\phi_t(x)\}_{[0,T]}$ is called {\em $(\lambda, T_0)^*$ quasi-hyperbolic} with respect to a splitting $N_x = E^N_x\oplus F^N_x$ and the scaled linear Poincar\'e flow $\psi^*_t$, if it is $(\lambda,T_0)$-forward contracting for $E_x^N$, $(\lambda,T_0)$-backward contracting for $F_x^N$, and satisfies:
\begin{equation}\label{e.dom}
\frac{\|\psi^*_{t_{i+1} - t_i}|_{\psi_{t_i}(E^N_x)}\|}{m(\psi^*_{t_{i+1} - t_i}|_{\psi_{t_i}(F^N_x)})} \le \lambda^2.
\end{equation}
\end{definition}

\begin{definition}\label{d.hyperbolictime}
Assume that the scaled linear Poincar\'e flow has a dominated splitting $E\oplus F$. 
A point $x$ is called a {\em $(\lambda,T_0)$-forward hyperbolic time for the bundle $E\subset N_x$}, if the infinite orbit $\phi_{[0,+\infty)}$ is $(\lambda,T_0)$-forward contracting. In this case the partition is taken as
$$
0=t_0<t_1<\cdots<t_n<\ldots, \mbox{\hspace{.2cm} where } t_{i+1}-t_i\in [T_0,2T_0], 
$$
and~\eqref{e.hyptime} is stated for all $k\in\NN$. Similarly,
$x$ is called a {\em $(\lambda,T_0)$-backward hyperbolic time for the bundle $F\subset N_x$}, if it is a forward hyperbolic time for the bundle $F$ and for the flow $-X$.  $x$ is called a {\em two-sided hyperbolic time}, if it is both a  forward and backward hyperbolic time.
\end{definition}

The following lemma states that two consecutive orbit segments that are both $(\lambda,T_0)$-forward contracting can be ``glued together'' to form a $(\lambda,T_0)$-forward contracting orbit segment. The same can be said for backward contracting orbit segments by considering the flow $-X$. 
\begin{lemma}\label{l.glue}
	The following statements hold:
\begin{enumerate}
	\item Assume that $\{\phi_t(x)\}_{[0,T_1]}$ is  $(\lambda,T_0)$-forward contracting for the bundle $E$, and  $\{\phi_t(\phi_{T_1}(x))\}_{[0,T_2]}$ is $(\lambda,T_0)$-forward contracting for the bundle $\psi_{T_1}(E)$. Then  $\{\phi_t(x)\}_{[0,T_1+T_2]}$ is  $(\lambda,T_0)$-forward contracting for the bundle $E$.
	\item Assume that the scaled linear Poincar\'e flow has a dominated splitting $E\oplus F$. 
	Let $\{\phi_t(x)\}_{[0,T_1]}$ be  $(\lambda,T_0)$-forward contracting for the bundle $E$, and assume that $\phi_{T_1}(x)$ is a $(\lambda,T_0)$-forward hyperbolic time for the bundle $E$. Then $x$ is  a $(\lambda,T_0)$-forward hyperbolic time for the bundle $E$. 
\end{enumerate}
\end{lemma}
The proof is standard and thus omitted.

By the classic work of Liao~\cite{Liao}, there exists $\delta =\delta(\lambda,T_0)>0$ such that if $x$ is a $(\lambda,T_0)$-backward hyperbolic time, then $x$ has unstable manifold with size $\delta \|X(x)\|$. Similarly, if $x$ is a $(\lambda,T_0)$-forward  hyperbolic time then it has stable manifold with size $\delta \|X(x)\|$.  In both cases, we say that {\em $x$ has unstable/stable manifold up to the flow speed}.

The next lemma can be seen as a $C^1$ version of the Pesin theory for flows. The proof can be found in~\cite[Section 4]{PYY}
\begin{lemma}\label{l.C1pesin}
Let $\mu$ be a hyperbolic measure for  the flow $\phi_t$. 
For almost every ergodic component $\tilde{\mu}$ of $\mu$ with respect to $f=\phi_1$, there are $L',\eta,T_0>0$ and a compact set $\Lambda_0\subset\supp\mu\setminus\Sing(X)$ with  positive $\tilde{\mu}$ measure, such that for every $x$ satisfying $f^n(x)\in \Lambda_0$ for  $n>L'$, the orbit segment $\{\phi_t(x)\}_{[0,n]}$ is $(\eta,T_0)^*$ quasi-hyperbolic with respect to the splitting $N_x = \pi_x(E_x)\oplus \pi_x(F_x)$ and the scaled linear Poincar\'e flow $\psi^*_t$.
\end{lemma}

Next. we introduce a shadowing lemma by Liao~\cite{Liao} for the scaled linear Poincar\'e flow. See~\cite{G02} and~\cite{PYY} for the current version.

\begin{lemma}\label{l.shadowing}
Given a compact set $\Lambda_0$ with $\Lambda_0 \cap \Sing(X) = \emptyset$ and $\eta\in(0,1),T_0>0$, for any $\varepsilon>0$ there exists $\delta>0$, $L>0$ and $\delta_0>0$, such that for any $(\eta,T_0)^*$ quasi-hyperbolic orbit segment $\{\phi_t(x)\}_{[0,T]}$ with respect to a dominated splitting $N_x = E_x \oplus F_x$ and the scaled linear Poincar\'e flow $\psi^*_t$, if $x,\phi_T(x) \in \Lambda_0$ with $d(x,\phi_T(x))<\delta$, then there exists a point $p$ and a $C^1$ strictly increasing function $\theta:[0,T] \to \RR$, such that
\begin{enumerate}[label=(\alph*)]
\item $\theta(0)=0$ and $|\theta'(t)-1|<\varepsilon$;
\item $p$ is a periodic point with $\phi_{\theta(T)}(p)=p$;
\item $d(\phi_t(x), \phi_{\theta(t)}(p))\le \varepsilon\|X(\phi_t(x))\|$, for all $t\in[0,T]$;
\item $d(\phi_t(x), \phi_{\theta(t)}(p))\le Ld(x,\phi_{T}(x))$;
\item $p$ has stable and unstable manifold with size at least $\delta_0$.
\item if $\Lambda_0\subset\Lambda$ for a chain recurrent class $\Lambda$, then $p\in\Lambda$.
\end{enumerate}
\end{lemma}

\subsection{Extended linear Poincar\'e flows}\label{s.extended}
Note that the (scaled) linear Poincar\'e flow is only defined at the regular points $M\setminus \Sing(X)$. To solve this issue, we introduce the {\em extended linear Poincar\'e flow}, which is a useful tool developed by Liao~\cite{Liao, Liao96} and Li et al~\cite{LGW} to study hyperbolic singularities.  

Denote by 
$$
G^1 = \{L: L \mbox{ is a 1-dimensional subspace of }T_xM, x\in M\}
$$
the Grassmannian manifold of $M$. Given a $C^1$ flow $\phi_t$, the tangent flow $\Phi_t$ acts naturally on $G^1$ by mapping each $L$ to $\Phi_t(L)$. 

Write $\beta:G^1\to M$ and $\xi: TM\to M$ the bundle projection. The pullback bundle of $TM$:
$$
\beta^*(TM) = \{(L,v)\in G^1\times TM: \beta(L)= \xi(v)\} 
$$
is a vector bundle over $G^1$ with dimension $\dim M$. The tangent flow $\Phi_t$ lifts naturally to $\beta^*(TM)$:
$$
\Phi_t(L,v) =(\Phi_t(L),\Phi_t(v)).
$$

Recall that the linear Poincar\'e flow $\psi_t$  projects the image of the tangent flow to the normal bundle of the flow direction. The key observation is that this projection can be defined not only w.r.t the bundle perpendicular to the flow, but to the orthogonal complement of any section $\{L_x:x\in M\}\subset G^1$.

To be more precise, given $\cL = \{L_x:x\in M\}$ we write 
$$
\cN_\cL = \{(L_x,v)\in \beta^*(TM): v \perp L_x\}.
$$
Then $\cN$, consisting of vectors perpendicular to $L$, is a sub-bundle of $\beta^*(TM)$ over $G^1$ with dimension $\dim M-1$. The {\em extended linear Poincar\'e flow} is then defined as 
$$
\psi_t: \cN_\cL\to \cN_\cL,\\
\psi_t(L_x,v) = \pi(\Phi_t(L_x,v)), 
$$
where $\pi$ is the orthogonal projection from fibres of $\beta^*(TM)$ to the corresponding fibres of $\cN$ along $\cL$.

If we consider the the map 
$$
\zeta: \Reg(X)\to G^1
$$
that maps every regular point $x$ to the unique $L_x\in G^1$ with  $\beta(L_x)=x$ such that $L_x$ is generated by the flow direction at $x$, then the extended linear Poincar\'e on $\zeta(\Reg(X))$ can be naturally identified with the linear Poincar\'e flow defined earlier. On the other hand, given any invariant set $\Lambda$ of the flow $\phi_t$, consider the set: 
$$
\tilde{\Lambda} = \overline{\zeta(\Lambda\cap \Reg(X))}.
$$ 
If $\Lambda$ contains no singularity, then $\tilde{\Lambda}$ can be seen as a natural copy of $\Lambda$ in $G^1$ equipped with the direction of the flow on $\Lambda$. If $\sigma\in\Lambda$ is a singularity, then $\tilde{\Lambda}$ contains  all the direction in $\beta^{-1}(\sigma)$ that can be approximated by the flow direction at regular points in $\Lambda$. 
The extended Poincar\'e flow restricted to $\tilde{\Lambda}$ can be seen as  the continuous extension of the linear Poincar\'e flow on $\Lambda$. The same treatment can be applied to the scaled linear Poincar\'e flow $\psi_t^*$.

\subsection{Classification of chain recurrent classes and singularities for generic star flows}\label{s.starflow}
In this subsection we recap the main result in~\cite{GSW} on $C^1$ generic star flows. We begin with the following classification on the singularities.

\begin{definition}
Let $\sigma$ be a hyperbolic singularity contained in a non-trivial chain recurrent class $C(\sigma)$. Assume that the Lyapunov exponents of $\sigma$ are:
$$
\lambda_1\le\cdots\le \lambda_s<0<\lambda_{s+1}\le\cdots\le\lambda_{\dim M}.
$$
\end{definition}
Write $\Ind(\sigma)=s$ for the stable index of $\sigma$. We say that
\begin{enumerate}
\item $\sigma$ is {\em Lorenz-like}, if $\lambda_s+\lambda_{s+1}>0$, $\lambda_{s-1}<\lambda_s$ (this implies that $e^{\lambda_s}<1$ is a real single  eigenvalue of $\Phi_t|_{T_\sigma M}$), and $W^{ss}(\sigma)\cap\{\sigma\}=\{\sigma\}$, where $W^{ss}(\sigma)$ is the stable manifold of $\sigma$ corresponding to $\lambda_1,\ldots,\lambda_{s-1}$; regular orbits in $C(\sigma)$ can only approach $\sigma$ along $E^{cu}(\sigma)$ cone, where $E^{cu}$ is the $\Phi_t$-invariant subspace correspond to $\lambda_s,\ldots,\lambda_{\dim M}$;
\item $\sigma$ is {\em reverse Lorenz-like},\footnote{In~\cite{GSW} and~\cite{CDYZ} both cases are called Lorenz-like. Here we distinguish between the two since our main argument are different in each case. See the proof of Lemma~\ref{l.multisingular} for more details.} if it is Lorenz-like for $-X$; in this case, regular orbits in $C(\sigma)$ can only approach $\sigma$ along $E^{cs}(\sigma)$ cone. See Figure~\ref{f.1}.
\end{enumerate}
\begin{figure}[h]
    \centering
    \includegraphics[scale=0.4]{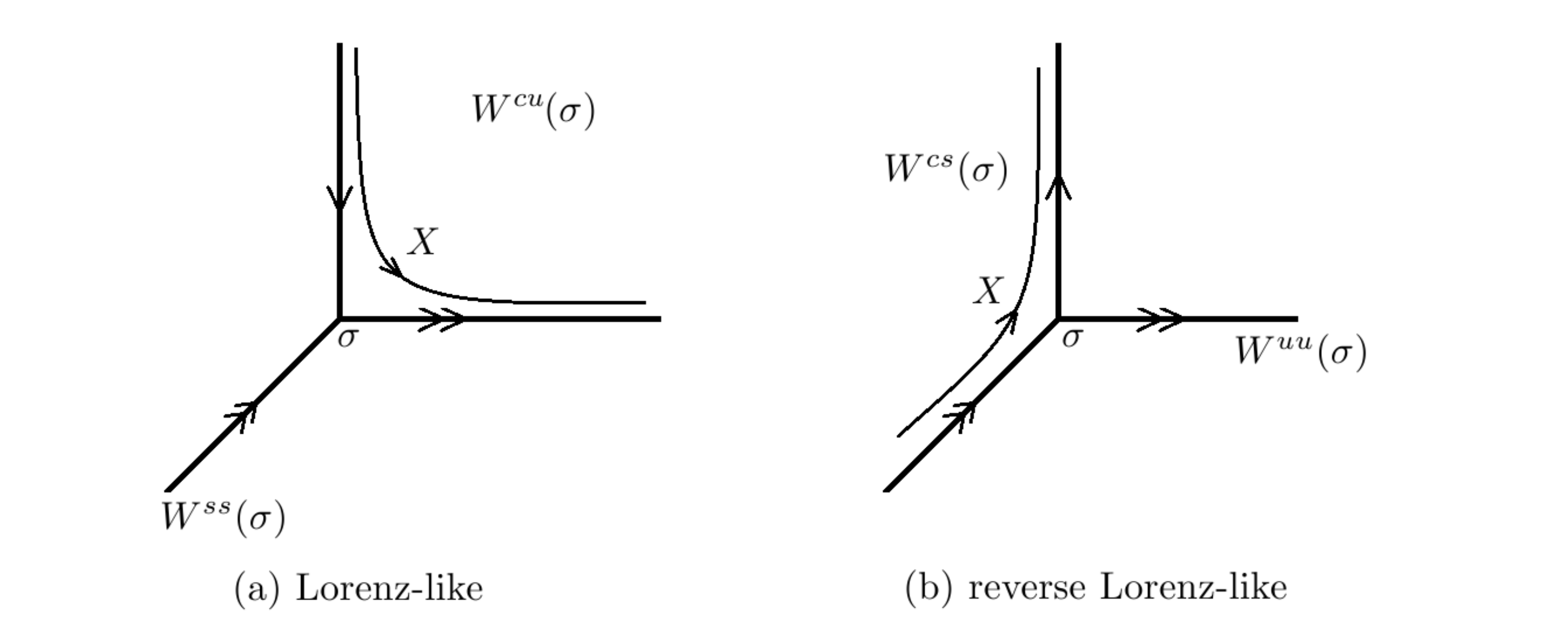}
    \caption{Lorenz-like and reverse Lorenz-like singularities}
    \label{f.1}
\end{figure}
Then it is shown in~\cite{LGW} and~\cite{GSW} that (all the theorems are labeled according to~\cite{GSW}):
\begin{itemize}
\item for a star vector field $X$, if a chain recurrent class $C$ is non-trivial, then every singularity in $C$ is either Lorenz-like or reverse Lorenz-like (Theorem 3.6); the original proof can be found in~\cite{LGW};
\item there exists a residual set $\cR\subset \mathscr{X}^1_*(M)$ such that for every $X\in\cR$, if a periodic orbit $p$ is sufficiently close to a singularity $\sigma$, then:
\begin{itemize}
\item when $\sigma$ is Lorenz-like, the index of the $p$ must be $\Ind(\sigma)-1$;
\item when $\sigma$ is reverse Lorenz-like, the index of the $p$ is $\Ind(\sigma)$ (Lemma 4.4);
\end{itemize}
furthermore, the dominated splitting on $\sigma$ induced by such periodic orbits coincides with the hyperbolic splitting on $\sigma$ (proof of Theorem 3.7);
\item For every chain recurrent class $C$ there exists an integer $\Ind_C>0$, such that every periodic orbit contained in a sufficiently small neighborhood of $C$ has  the same stable index which equals $\Ind_C$ (Theorem 5.7);
\item combine the previous two results, we see that all the singularity in $C$ has index either $\Ind_C+1$ (in which case it must be Lorenz-like) or $\Ind_{C}$ (reverse Lorenz-like);
\item if all the singularities in $C$ are Lorenz-like, then $C$ is sectional hyperbolic (Theorem 3.7); if all the singularities in $C$ are reverse Lorenz-like, then $C$ is sectional hyperbolic for $-X$ (Theorem 3.7);
\item if $C$ contains singularity with different indices (note that they can only differ by one), then there is no sectional hyperbolic splitting on $C$; one of such  examples was constructed by Bonatti and da Luz~\cite{BD, daLuz}.
\end{itemize}

\section{Flow orbits near singularities}\label{ss.singularity}
This section contains some general results on hyperbolic singularities of $C^1$ vector fields (Section~\ref{s.3.1}), and on Lorenz-like singularities for star flows (Section~\ref{s.3.2}). The key results are Lemma~\ref{l.turning.hyperbolic} and~\ref{l.timeincone}, which state that the time near a singularity where the orbit is ``make the turn'' is bounded from above. For Lorenz-like singularities of star flows, we prove in Lemma~\ref{l.cc} that for a periodic orbit $\Orb(p)$ approaching a singularity $\sigma$ while exhibiting backward hyperbolic times,  the unstable manifold of $\Orb(p)$ must transversally intersect with $W^s(\sigma)$. Then Lemma~\ref{l.cc2} deals with the case where a sequence of forward hyperbolic times approaches a Lorenz-like singularity. These two results will allow us to show in the next section that such periodic orbit must be in the same chain recurrent class as $\sigma$.

\subsection{Flow orbits near hyperbolic singularities}\label{s.3.1}
In this section we will establish some geometric properties for flow orbits in a small neighborhood of a hyperbolic singularity. Our result applies to all $C^1$ vector fields $X$ which are  not necessarily star.

For this purpose, let $\sigma$ be a hyperbolic singularity with the hyperbolic splitting $E^s_\sigma\oplus E^u_\sigma$. Without loss of generality, we can think of $\sigma$ as the origin in $\RR^n$, and assume that $E^s_\sigma$ and $E^u_\sigma$ are perpendicular (which is possible if one changes the metric). In particular, we will assume that $E^s_\sigma = \RR^{s}$ is the $s$-dimensional subspace of $\RR^n$ with the last $\dim M-s$ coordinates being zero. Here $s=\dim E^s_\sigma$ is the stable index of $\sigma$. Similarly, $E^u_\sigma$ is the subspace of $\RR^n$ where the first $s$ coordinates are zero.  As before we will write $f=\phi_1$ for the time-one map of the flow.

Since the vector filed $X$ is $C^1$, we can take a  neighborhood $U=B_r(\sigma)$ with $r$ small enough, such that: 
\begin{itemize}
	\item the flow in $U$ can be written as 
	\begin{equation}\label{e.c1flow}
		\phi_t(x) = e^{At}x+C^1\mbox{ small perturbation},
	\end{equation}
	where $A$ is a matrix  no eigenvalue on the imaginary axis;
	\item for $x\in U$, the tangent map $Df_x = \Phi_1|_{T_xM}$ are small perturbations of the hyperbolic matrix $e^A$, with eigenvalues bounded away from $1$.
\end{itemize}
For each $x\in U$, denote by  $x^s$ its distance to $E^u_\sigma$ and $x^u$ its distance to $E^{s}_\sigma$. Then for every $\alpha>0$ small, we define the {\em $\alpha$-cone on the manifold}, denote by $D^{i}_\alpha(\sigma)$, $i=s,u$, as follows:
$$
D^s_\alpha(\sigma) = \{x\in U: x^u<\alpha x^s\},\hspace{1cm}D^u_\alpha(\sigma) = \{x\in U: x^s<\alpha x^u\}.
$$

Note that the hyperbolic splitting $T_\sigma M = E^s_\sigma\oplus E^u_\sigma$ can be extended to $U$ in a natural way: for each $x\in U$, put $E^s(x)$ as the $s$-dimensional hyperplane that is parallel to $E^s_\sigma$; the same can be done for $E^u(x)$. This allows us to consider the $\alpha$-cones $C_\alpha(E^i)$, $i=s,u$, on the tangent bundle as defined in Section~\ref{s.2.1}. 
The next lemma easily follows from the smoothness of the vector field $X$ and the hyperbolicity of $\sigma$:

\begin{lemma}\label{l.cones}
There exists $L\ge1$, such that for all $\alpha>0$ small enough,
\begin{enumerate}
	\item for every $x\in \Cl(D^s_\alpha(\sigma))$, we have $X(x)\in C_{L\alpha}(E^s)$;
	\item for every $x\in U$, if $X(x)\in C_\alpha(E^s)$, we have $x\in D^s_{L\alpha}(\sigma)$.  
\end{enumerate}
Moreover, the same holds for $D^u_\alpha(\sigma)$ and $C_\alpha(E^u)$.
\end{lemma}

Let us fix some $\alpha>0$ small enough that will be specified at the end of this section. Note that if $x\in U\setminus (D^s_\alpha(\sigma)\cup D^u_\alpha(\sigma))$, we lose  control on the direction of $X(x)$. One can think of the region  $ U\setminus (D^s_\alpha(\sigma)\cup D^u_\alpha(\sigma))$ as the place where the flow is `making the turn' from the $E^s$ cone to the $E^u$ cone. The next lemma states that the time that an orbit segment spend in this region is uniformly bounded. To this end, we write, for each $x\in U$, 
$$
t^+(x) = \sup\{t>0:\phi_{[0,t]}(x)\subset U\},\hspace{.5cm}t^-(x) = \sup\{t>0:\phi_{[-t,0]}(x)\subset U\}.
$$
Then the orbit segment $\phi_{(-t^-,t^+)}(x)$ contains $x$ and is contained in $U$.
With slight abuse of notation, we will frequently drop the depends of $t^\pm(x)$ on $x$.

\begin{lemma}\label{l.turning.hyperbolic}
Let $\sigma$ be a hyperbolic singularity for a $C^1$ vector field $X$. Then for every $\alpha>0$ small enough, there exists $T_\alpha>0$ such that for every $r>0$ small enough and every $x\in U = B_r(x)$, the set 
$$
T(x):=\{t\in(-t^-,t^+):\phi_t(x)\notin D^s_\alpha(\sigma)\cup D^u_\alpha(\sigma)\}
$$	
has Lebesgue measure bounded by $T_\alpha$.
\end{lemma}
\begin{proof}
Recall that the Lebesgue measure on the interval $(-t^-, t^+)$ corresponds to the length of open intervals. Below we will prove that the set in question is contained in a subinterval of $(-t^-, t^+)$ whose length is bounded from above.

We take a small neighborhood $x\in V\subset U$, such that for every $y\in V$ it holds
$$
\phi_{(-t^-(y),t^+(y))}(y)\cap D_{\alpha}^*(\sigma)\ne\emptyset, *=s,u.
$$
Note that if a orbit segment $\phi_{(-t^-,t^+)}(x)$ does not intersect with $V$, then $t^-+t^+$ must be bounded. Therefore we only need to prove the lemma for orbit segments that  intersect with $V$.

Shrinking $V$ if necessary, we may assume that 
$$
\phi_{(-t^-,t^+)}(x)\cap V\ne\emptyset \implies \phi_{-t^-}(x)\in D^s_\alpha(\sigma), \phi_{t^+}(x)\in D^u_\alpha(\sigma).
$$
We will also assume, by changing to a different point on the same orbit segment if necessary, that $x\in V\setminus (D^s_\alpha(\sigma)\cup D^u_\alpha(\sigma))$.
For such an orbit segment $\phi_{(-t^-,t^+)}(x)$, define 
$$t^s=t^s(x) = \sup\{t>0: \phi_{(-t^-,-t)}(x)\subset D^s_\alpha(\sigma)\},
$$
and  
$$t^u = t^u(x) = \sup\{t>0: \phi_{(t,t^+)}(x)\subset D^u_\alpha(\sigma)\}.$$ Clearly we have $T(x)\subset (-t^s, t^u)$. Below we will show that $t^s + t^u$ is bounded from above.

Writing $x^s = \phi_{-t^s}(x)$, $x^u = \phi_{t^u}(x)$, Lemma~\ref{l.cones}(2) shows that 
$$
X(x^s)\notin C_{\alpha/L}(E^s),\,\, X(x^u)\notin C_{\alpha/L}(E^u).
$$
In particular\footnote{Following our notation earlier, $X(x^*) = X(x^*)^s+X(x^*)^u\in E^s\oplus E^u$.},
\begin{equation}\label{e.3.2}
\frac{||X(x^s)^u||}{ ||X(x^s)^s||}>\alpha/L,\,\,\,\, \frac{||X(x^u)^u||}{||X(x^u)^s||}<L/\alpha.
\end{equation}

On the other hand, by the hyperbolicity of $\sigma$, there is $\lambda_0>1$ independent of $\alpha$ such that for each $x\in U\cap \phi_{-1} (U)$, and for every $v\in T_x M$ it holds that 
$$
\frac{||\Phi_1(v)^u||}{||\Phi_1(v)^s||} > \lambda_0\frac{||v^u||}{||v^s||}. 
$$
Combine this with ~\eqref{e.3.2} and the observation that $\Phi_{t^s+t^u}(X(x^s)) = X(x^u)$, we obtain 
$$
\lambda_0^{t^s+t^u} < \frac{L^2}{\alpha^2},
$$
which implies that 
$$
t^s+t^u<\frac{2\log L - 2\log \alpha}{\log\lambda_0}:= T_\alpha.
$$
Also note that $T_\alpha$ can be made uniform for $r>0$ small enough, since $L$ and $\lambda_0$ can be chosen independent of $r$. This concludes the proof of the lemma.
\end{proof}

Now let us look at this lemma from the perspective of invariant measures. Setting for $i=s,u$,
$$\cL^i(\sigma)=\{L\in G^1: \beta(L)=\sigma, L\mbox{ is parallel to } E^i \},$$ 
then $\cL^i$ are invariant under $\Phi_t|_{\beta^{-1}(\sigma)}$ (note that this is the tangent flow on $G^1$). Furthermore, the hyperbolicity of $\sigma$ implies that $\cL^{s}$ is a repelling set while $\cL^u$ is an attracting set.

Next we take a sequence of points $\{x_i\}\subset U$ with $x_i \to \sigma$ as $i\to\infty$. To simply notation, we will write $t^\pm_i = t^\pm(x_i)$. Note that $t^\pm_i\uparrow +\infty$. For every $\varepsilon>0$ small enough, the time that the orbit segments $\phi_{(-t^-_i,t^+_i)}(x_i)$ spend in the region $U\setminus B_\varepsilon(\sigma)$ is uniformly bounded in $i$. As a result, the empirical measures supported on these orbit segments behaves trivially:%
\begin{equation}\label{e.empirical}
\nu_i=\frac{1}{t^-_i+t^+_i}\int_{-t^-_i}^{t^+_i}\delta_{\phi_s(x_i)}\,ds\xrightarrow{i\to\infty}\delta_\sigma,
\end{equation}
where $\delta_\sigma$ is the atomic measure on $\sigma$.

On the other hand,  the map $\zeta: \Reg(X)\to G^1$ defined in Section~\ref{s.extended} lifts any measure $\mu$ on $M$ with $\mu(\Sing)=0$ to a measure $\zeta_*(\mu)$ on $G^1$.  Now consider the lifted empirical measures:
\begin{equation}\label{e.measurelift}
\tilde{\nu}_i = \zeta_*(\nu_i).
\end{equation}
If we take any weak-* limit $\tilde{\mu}$ of $\{\tilde{\nu}_i\}$ (the limit exists since $G^1$ is compact), $\tilde{\mu}$ must be invariant under $\Phi_t$ and is supported on $\cL^{s}\cup \cL^u\subset \beta^{-1}(\sigma)$ since $\cL^{s}\cup \cL^u$ contains the non-wandering set of $\Phi_t|_{\beta^{-1}(\sigma)}$. 
Write $\cU_\alpha^* = \overline{\zeta(D^*_\alpha(\sigma))}$ for $*=s,u$.
Observe that by Lemma~\ref{l.cones}, $\cU_\alpha^*$ each contains a neighborhood of $\cL^*(\sigma)$ in $G^1$, $*=s,u$. Furthermore, we have $\cU_\alpha^s\cap\cU_\alpha^u=\emptyset$.
Combine this  with Lemma~\ref{l.turning.hyperbolic}, we obtain the following lemma:

\begin{lemma}\label{l.turning.measure}
	For all $\alpha>0$ small, we have  $\tilde{\nu}_i(\cU_\alpha^s\cup \cU_\alpha^u)\to 1$ as $i\to N$.
\end{lemma}

The next lemma states that the time that orbit segments $\phi_{(-t^-_i,t^+_i)}(x_i)$ spend in $D^s_\alpha(\sigma)$ and $D^u_\alpha(\sigma)$ are comparable:

\begin{lemma}\label{l.timeincone}
	There is $a\in(0,\frac12)$ independent of $\alpha$, such that for every sequence $\{x_i\}\subset U$ with $x_i\to\sigma$ and every weak*-limit $\tilde{\mu}$ of the empirical measure $\tilde{\nu}_i$ defined using~\eqref{e.empirical} and \eqref{e.measurelift}, we have
	$$
	\tilde{\mu}(\cU_\alpha^s)>a \mbox{ and } \tilde{\mu}(\cU_\alpha^u)>a.
	$$
\end{lemma}
\begin{proof}
	We will show that for every orbit segment $\phi_{(-t^-_i,t^+_i)}(x_i)$, the time it spends in $D^*_\alpha$, $* = s, u$ are comparable, with a ratio that is uniform in $i$ and $\alpha$.

	First, note that since the vector field is $C^1$, the flow speed is a Lipschitz function of $d(x,\sigma)$:  there is $0<C_1<C_2$ such that 
	$$
	\frac{\|X(x)\|}{d(x,\sigma)}\in (C_1,C_2).
	$$
	To simplify notation, we write 
	$$x_{e,i} = \phi_{-t^-_i}(x_i) \mbox{, and }x_{l,i} = \phi_{t^+_i}(x_i) $$ 
	for the end points of $\phi_{(-t^-_i,t^+_i)}(x_i)$ that enter and leave the neighborhood $U$. 
	By our construction, $x_{e,i},x_{l,i}\in\partial U = \partial B_r(x)$.
	As a result, for every $i$ it holds 
	$$
	\frac{\|X(x_{e,i})\|}{\|X(x_{l,i})\|}\in\left(\frac{C_1}{C_2}, \frac{C_2}{C_1}\right).
	$$
	
	Denote by $t^0_i\in (-t^-_i,t^+_i)$ the time such that the point  $x^0_i = \phi_{t_0}(x_i)$ satisfies
	$(x^0_i)^s = (x^0_i)^u$. One could think of $x^0_i$ as the point on the orbit segment $\phi_{(-t^-_i,t^+_i)}(x_i)$ where the flow speed is the lowest.  We parse each orbit segment $\phi_{(-t^-_i,t^+_i)}(x_i)$ into three sub-segments (recall the definition of $t^s(x_i)$ and $t^u(x_i)$ in  Lemma~\ref{l.turning.hyperbolic}. To simplify notation we will write $t_i^* = t^*(x_i)$, $*=s,u$):
	\begin{itemize}
		\item write $x^s_{i} = \phi_{t^s_i}(x_i)$ for the point on $\phi_{(-t^-_i,t^+_i)}(x_i)$ that is on the boundary of $D^s_\alpha(\sigma)$; then the orbit from $x_{e,i}$ to $x^s_{i}$ is contained in $D^s_\alpha(\sigma)$;
		\item  write $x^u_{i}=\phi_{t^u_i}(x_i)$ for the point on $\phi_{(-t^-_i,t^+_i)}(x_i)$ that is on the boundary of $D^u_\alpha(\sigma)$; then the orbit from $x^u_i$ to $x_{l,i}$ is contained in $D^u_\alpha(\sigma)$;
		\item the orbit segment from $x^s_i$ to $x^u_i$ is outside $D^*_\alpha(\sigma)$, $*=s,u$; by Lemma~\ref{l.turning.hyperbolic}, $t^u_i-t^s_i\le T_\alpha$. 
	\end{itemize}
	Note that $x^0_i$ is contained in the  orbit segment from $x^s_i$ to $x^u_i$. Since the flow time from  $x^0_i$ to $x^\pm_i$ is bounded by $T_\alpha$ and the flow is $C^1$, we obtain
	$$
	\frac{\|X(x^u_i)\|}{\|X(x^s_i)\|}  = \frac{\|X(x^u_i)\|}{\|X(x^0_i)\|} \frac{\|X(x^0_i)\|}{\|X(x^s_i)\|} \in\left( ||\Phi_{T_\alpha}||^2, ||\Phi_{T_\alpha}||^{-2} \right).
	$$
	
	For the orbit segment from $x_{e,i}$ to $x^s_{i}$, Lemma~\ref{l.cones}(1) shows that $X(x)\in C_{L\alpha}(E^s)$ for each $x$ in this orbit segment. Since the flow speed is uniformly exponentially contracting in $C_{L\alpha}(E^s)$ provided that $\alpha$ and $r$ are small enough, we see that the time length of this orbit segment satisfies
	$$
	t^-_i-t^s_i =\mathcal{O}(\log\frac{\|X(x_{e,i})\|}{\|X(x_i^-)\|}).
	$$ 
	Similarly, 
	$$
	t^+_i - t^u_i = \mathcal{O}(\log\frac{\|X(x_{l,i})\|}{\|X(x^u_i)\|}).
	$$
	Then the ratio is 
	$$
	\frac{t^-_i-t^s_i}{t^+_i - t^u_i} =\mathcal{O}\left(\frac{\log\|X(x_{e,i})\|-\log\|X(x_i^-)\|}{\log\|X(x_{l,i})\|-\log\|X(x^u_i)\|}\right) =\mathcal{O}\left(\frac{\log\|X(x_i^-)\|}{\log\|X(x^u_i)\|}\right)=\mathcal{O}(1),
	$$
	where in the last equality we use the elementary fact that if $a_i\to0,b_i\to 0$ such that $a_i/b_i$ is bounded from above and away from zero, then $\log a_i/\log b_i\to 1.$
	
	Finally, note that even though the ratio $	\frac{\|X(x^u_i)\|}{\|X(x^s_i)\|} $ depends on $\alpha$, $\frac{t^-_i-t^s_i}{t^+_i - t^u_i}$ only depends on the exponential contracting/expanding rate in  $C_{L\alpha}(E^*)$, which can be made uniform for $\alpha$ small enough. 
	This finishes the proof of the lemma.
\end{proof}


We conclude this subsection with the following lemma, which will be used later to create transverse intersection between the unstable manifold of a periodic orbit and the stable manifold of the singularity $\sigma$. As before, $s$ is the stable index of $\sigma$.

\begin{lemma}\label{l.intersection}
	For each $\beta>0$ small and $\delta>0$, there is $\alpha_0>0$ such that for all $\alpha<\alpha_0$ and for every point $x\in D^s_\alpha(\sigma)$, let  $W(x)$ be a $(\dim M-s)$-dimensional submanifold that contains $x$ and is tangent to $C_\beta(E^u)$. If $\diam W(x)>\delta\|X(x)\|$, then $W(x)\pitchfork W^s(\sigma)\ne\emptyset$.  
\end{lemma}
\begin{proof}
	Since the flow speed at $x$ is a Lipschitz function of $d(x,\sigma)$, we see that 
	$$\diam W(x)>C_1\delta d(x,\sigma)>C_1C'\delta x^s$$ 
	for some $C'>0$. Here $C_1$ is the same constant as in the previous lemma. 
	
	On the other hand, since $W(x)$ is tangent to the $\beta$-cone of $E^u$, there is $C''>0$ such that if $\diam W(x)>C'' x^u=C''d(x,W^s(\sigma))$, we must have $W(x)\pitchfork W^s(\sigma)\ne\emptyset$.

	Since $x^u<\alpha x^s $ in the cone $D^s_\alpha$, the choice of $\alpha< \alpha_0:=C_1C'\delta/C''$ guarantees that $\diam W(x)>CC'\delta x^s>C'' x^u$, therefore $W(x)\pitchfork W^s(\sigma)\ne\emptyset$. This concludes the proof of the lemma.
\end{proof}

\subsection{Near Lorenz-like singularities}\label{s.3.2}
Now we turn our attention to Lorenz-like singularities for star flows.
Assume that $\sigma$ is a Lorenz-like singularity contained in a non-trivial chain recurrent class $C=C(\sigma)$. The discussion below applies to reverse Lorenz-like singularities if one considers the flow $-X$.

Let $E_\sigma^{cs}\oplus E_\sigma^u$ be the hyperbolic splitting on $T_\sigma M$. We will write $E_\sigma^{ss}$ the subspace of $T_\sigma M$ corresponding to the exponents $\lambda_1,\ldots,\lambda_{s-1}$ and $E_\sigma^c$ the subspace of $T_\sigma M$ corresponding to the exponent $\lambda_{s}$. Then $E_\sigma^{ss}\oplus E_\sigma^c = E_\sigma^{cs}$.

The discussion in the previous sub-section applies to $\sigma$ without any modification (note that this time, we change the notation of $E^s$ to $E^{cs}$ and $\cL^s$ to $\cL^{cs}$). 
Furthermore, we can think of $\sigma$ as the origin in $\RR^n$ with three bundles $E_\sigma^{ss},E_\sigma^c$ and $E_\sigma^u$ perpendicular to one another (which is possible if one changes the metric). These bundles can be naturally extended to $U=B_r(x)$ as before. As a result, the cone field $C_\alpha(E^*)$ can be defined for $*=ss,c,u, cu$ and $cs$. 
The same can be said about the cones  on the manifold, $D_\alpha^*(\sigma)$.

It is proven in~\cite{MPP04},~\cite{LGW} and \cite{GSW} that $W^{ss}(\sigma)\cap\C(\sigma)=\{\sigma\}$. 
Furthermore, it is shown that if the orbit segment is taken inside $C(\sigma)$ (or if the orbit segment belongs to a periodic orbit that is sufficiently close to $C(\sigma)$), then it can only approach the singularity  $\sigma$ along the one-dimensional subspace $E^c$ in the following sense: write
$$
\cL^c = \{L\in G^1: \beta(L) = \sigma, L \mbox{ is parallel to } E^c\}, 
$$
then $\cL^c\subset \cL^{cs}$ consists of a single point in $G^1$. If we take $x_i\in C(\sigma)$ with $x_i\to\sigma$ and define the empirical measure $\nu_i$ and its lift $\tilde{\nu}_i$ according to~\eqref{e.empirical} and~\eqref{e.measurelift}, then any weak*-limit $\tilde{\mu}$ of $\{\tilde{\nu}_i\}$ must satisfy
$$
\supp\tilde{\mu} = \cL^c\cup\cL^u. 
$$
This observation leads to the following lemma, which is an improved version of Lemma~\ref{l.turning.measure} and~\ref{l.timeincone}.

\begin{lemma}
	Let $\{x_i\}\subset C(\sigma)$, then the conclusion of Lemma~\ref{l.turning.measure} and Lemma~\ref{l.timeincone} remain true with $\cU^s(\alpha)$ replaced by a neighborhood  $\cU^c(\alpha)$ of $\cL^c$ in $G^1$. The same can be said if we take $\{p_i\}$ to be periodic points with $p_i\to\sigma$ but not necessarily in $C(\sigma)$.
\end{lemma}
The proof remains unchanged and is thus omitted.

Now let us describe the hyperbolicity of the periodic orbits close to $\sigma$. 
Recall that for a regular point $x$, $N_x\subset T_xM$ is the orthogonal complement of the flow direction $X(x)$.
Since $X$ is a star vector field, every periodic orbit is hyperbolic. It is proven in~\cite[Theorem 3.7]{GSW} if $p_n$ is a sequence of periodic points near $C(\sigma)$ and approaches $\sigma$, then the hyperbolic splitting $E^s\oplus E^{cu}$ (with $E^{cu} = <X>\oplus E^u$) on $\Orb(p_n)$ extends to a dominated splitting on $\sigma$, which coincides with the splitting $E^{ss}_\sigma\oplus E_\sigma^{cu}$ where $E_\sigma^{cu} = E_\sigma^c\oplus E_\sigma^u$. Since we assume that $E^{ss}$ is perpendicular to $E^{cu}$, and the flow direction is tangent to the  cone $C_{L\alpha}(E^{c})$ as the orbit approaches the singularity $\sigma$, it follows that for $n$ large enough, the local stable manifold $W^s(p'_n)$ (which has dimension $s-1$) for $p'_n\in\Orb(p_n)\cap D^{cs}_\alpha(\sigma)$ 
is tangent to a $E^{ss}$ cone.

\begin{remark}\label{r.transverse}
It is tempting to argue that for $n$ large enough, $W^u(p'_n)$ must intersect transversally with $W^{cs}(\sigma)$. However, this is not necessarily the case: as $p_n$ gets closer to the singularity $\sigma$, the size of the invariant manifolds of $p$ will shrink with rates proportional to the flow speed, as observed by Liao~\cite{Liao}. As a result, even if we have a sequence of periodic points $p_n\to p\in W^{cs}(\sigma)$, there is still no guarantee that the unstable manifold of $p_n$ will intersect with $W^{cs}(\sigma)$.  To solve this issue, we consider the hyperbolic times of $p_n$ as defined in Definition~\ref{d12}. But before that, let us first estimate the hyperbolicity of the orbit segment inside $D^{cs}_\alpha$ and $D^u_\alpha$.
\end{remark}

\begin{lemma}\label{l.sgood}
Let $\sigma$ be a Lorenz-like singularity, then for $r$ small enough, for every orbit segment $\phi_{[0,T]}(x)\subset B_r(\sigma)\cap C(\sigma)$, the scaled linear Poincar\'e flow $\psi^*_t|_{E^{ss}_x}$ is uniformly contracting, where $E^{ss}_x\subset N_x$ is the stable subspace for the scaled linear Poincar\'e flow. The same holds true for every periodic orbit segment $\phi_{[0,T]}(p)\subset B_r(\sigma)$ but not necessarily in $C(\sigma)$.
\end{lemma}
\begin{proof}
We only need to estimate
$$
\psi^*_t(v)=\frac{\|X(x)\|}{\|X(\phi_t(x))\|}\psi_t|_{E^{ss}_x}(v),
$$
for $v\in E^{ss}_x$.

We take $\varepsilon>0$ small enough such that $\lambda_{s-1}+2\varepsilon < \lambda_s-\varepsilon <0 $ (recall that $\lambda^s<0$ is the largest negative exponent of $\sigma$). If we take $r>0$ small enough, then inside $B_r(\sigma)$ we have 
$$
\|\psi_t|_{E^{ss}_x}(v)\|\le e^{(\lambda_{s-1}+\varepsilon)t}\|v\|.
$$
On the other hand, for $\frac{\|X(x)\|}{\|X(\phi_t(x))\|}$ we have (in the worst case scenario, where the flow direction is tangent to the $E^c$ cone):
$$
\|X(\phi_t(x))\|\ge e^{\lambda_s-\varepsilon }\|X(x)\|.
$$
Indeed the flow speed is expanding while the direction is tangent to the $E^u$ cone. While the orbit is neither  in $D^{cs}_\alpha$  nor in the $D^u_\alpha$, we lose all the estimate. However, the length of such orbit segment is uniformly bounded due to Lemma~\ref{l.turning.hyperbolic}, and can be safely ignored.

As a result, we get
$$
\|\psi^*_t(v)\|\le e^{(\lambda_{s-1}+\varepsilon -\lambda_s+\varepsilon)t}\|v\|\le e^{-\varepsilon t}\|v\|,
$$
and conclude the proof of the lemma.
\end{proof}

On the unstable subspace  $E^u_x\subset N_x$, the situation is different: when the orbit of $x$ moves away from $\sigma$, it can only do so along $D^u_\alpha(\sigma)$. As a result, one loses the hyperbolicity along the $E^u_x$ direction. On the other hand, when the orbit approaches $\sigma$, the orbit segment will be `good' (in the sense that it is quasi-hyperbolic according to Definition~\ref{d11}) as long as the flow direction is tangent to the $E^c$ cone. This is summarized in the next lemma:
\begin{lemma}\label{l.u}
Let $\sigma$ be a Lorenz-like singularity, then there exists $\lambda\in(0,1),T_0>0$, $\alpha_1>0$ and $r>0$, such that if $\alpha<\alpha_1$ and $x$ is a periodic orbit such that the orbit segment $\phi_{[0,T]}(x)$ is contained in $B_r(\sigma)\cap \Cl(D^{cs}_\alpha(\sigma))$, then the orbit segment 
is $(\lambda,T_0)$-backward contracting. If the orbit segment is in $\Cl(D^u_\alpha(\sigma))$, then it does not have any sub-segment that is backward contracting.
\end{lemma}

\begin{proof}
\noindent {\em Case  1.  $\phi_{[0,T]}(x)\subset  \Cl(D^{cs}_\alpha(\sigma))$.}

To simplify notation we write $y=\phi_T(x)$ for the endpoint of the orbit segment (note that it is the starting point for the same orbit segment under $-X$). Take any $t\in [0,T]$, we will estimate  
$$
\psi^*_{-t}(v)=\frac{\|X(y)\|}{\|X(\phi_{-t}(y))\|}\psi_{-t}(v),
$$
for $v\in E^u_{y}$. 

Recall that $\lambda_{s+1}$ is the smallest positive exponent of $\sigma$
. 
Like in the previous lemma, we take $\varepsilon>0$ small such that $\lambda_s+\varepsilon<0$, $\lambda_{s+1}-\varepsilon>0$, therefore $-\lambda_{s+1}+\lambda_s+2\varepsilon<0$.

Then we can take $r>0$ and $\alpha_1>0$ small enough such that all the lemmas in Section~\ref{s.3.1} hold; furthermore,
$$
\|\psi_{-t}(v)\|\le e^{-(\lambda_{s+1}-\varepsilon)t}\|v\|,
$$
and 
$$
\|X(\phi_{-t}(y))\|\ge e^{-(\lambda_s+\varepsilon)t}\|X(y)\| ,
$$
since the orbit segment is in $\Cl(D^{cs}_\alpha(\sigma))$.
This gives
$$
\|\psi^*_{-t}(v)\|\le  e^{(-\lambda_{s+1}+\varepsilon+\lambda_s+\varepsilon)t}\|v\| = e^{(-\lambda_{s+1}+\lambda_s+2\varepsilon)t}\|v\|,
$$
which shows that the orbit segment is backward contracting.

\noindent {\em Case 2. $\phi_{[0,T]}(x)\subset \Cl(D^u_\alpha(\sigma))$.}

We take any orbit segment $\phi_{[0,T]}(x)$ in $\Cl(D^u_\alpha(\sigma))$ and take any $y\in \phi_{[0,T]}(x)$. By Lemma~\ref{l.cones} the flow direction is almost parallel to $E^u$ if we take $\alpha_1$ small enough. As a result, we can take  $v\in E^u_y$ such that $v$ is almost parallel to $E^c(\sigma)$. For such $v$ we have
$$
\|\psi_{-t}(v)\|\ge e^{-(\lambda_{s}+\varepsilon)t}\|v\|,
$$
and the flow direction satisfies
$$
\|X(\phi_{-t}(y))\|\le e^{(-\lambda_{s+1}+\varepsilon)t}\|X(y)\|.
$$
This shows that 
$$
\|\psi^*_{-t}(v)\|\ge e^{(-\lambda_s-\varepsilon+\lambda_{s+1}-\varepsilon)t}\|v\|=e^{(-\lambda_s+\lambda_{s+1}-2\varepsilon)t}\|v\|.
$$
For $\varepsilon$ small enough, $-\lambda_s+\lambda_{s+1}-2\varepsilon>0$. As a result, $\psi^*_{-t}$ will never be contracting as long as the orbit segment is contained in $\Cl(D^u_\alpha(\sigma))$. Therefore $\phi_{[0,T]}(x)$ does not contain any sub-segment that is backward contracting.
\end{proof}
\begin{remark}
The previous two lemmas have similar formulations for reverse Lorenz-like singularities. 
\end{remark}

Recall the definition of $t^\pm(x)$ from the previous section.
Next, we introduce the main lemmas in this section, which will enable us to solve the issue mentioned in Remark~\ref{r.transverse} and create transverse intersection between the invariant manifolds of $p$ with points in $C$.

\begin{lemma}\label{l.main}
Let $\sigma$ be a Lorenz-like singularity. Then for $r>0$ small enough, for every  $\lambda\in(0,1), T_0>0$ there exists $\alpha_2>0$ such that for all $\alpha<\alpha_2$, if $y\in B_r(\sigma)\cap\Cl(D_{\alpha}^{cs}(\sigma))$ is a $(\lambda, T_0)$-backward hyperbolic time on its orbit, then $W^u(y)\pitchfork W^{cs}(\sigma)\ne\emptyset$.
\end{lemma}
\begin{proof}
Since $y$ is a $(\lambda, T_0)$-backward hyperbolic time, according to the classic work of Liao~\cite{Liao}, there is $\delta>0$ such that $W^u(y)$ has  size $\delta \|X(y)\|$, tangent to $C_\beta(E^u)$ for some $\beta>0$ (in fact, tangent to $C_\beta(E^{uu}_{y})$ where $E^{uu}_{y}$ is the unstable subspace in $N_{y}$; however, we may assume that $E^{uu}_{y}$ and $E^u$ are almost parallel as long as the flow orbit remains in the cone $\Cl(D^{cs}_\alpha(\sigma))$). Let $\alpha_2$ be given by Lemma~\ref{l.intersection} for such $\beta$ and $\delta$, we see that  $W^u(y)\pitchfork W^{cs}(\sigma)\ne\emptyset$.

\end{proof}

\begin{lemma}[Backward hyperbolic times near $\sigma$]\label{l.cc}
Let $\sigma$ be a Lorenz-like singularity, and $\{p_n\}$ be a sequence of periodic points with $p_n\to \sigma$.  For $\lambda\in(0,1), T_0>0$ and for $\alpha\in(0,\min\{\alpha_1,\alpha_2\}) $, 
assume that the set 
$$
H_n=\{t\in(-t^-_n,t^+_n):\phi_t(p_n) \mbox{ is a }(\lambda,T_0)\mbox{-backward hyperbolic time.}\}
$$
has positive density: there exists $a>0$ such that for every $n$,
$$
\nu_n(\{\phi_t(p_n):t\in H_n\})>a>0,
$$
where $t^\pm_n$ and $\nu_n$ are taken according to~\eqref{e.empirical}. Then there exists $N>0$ such that 
$$
W^u(\Orb(p_n))\pitchfork W^{cs}(\sigma)\ne\emptyset, \mbox{ for all } n > N.
$$
\end{lemma}
\begin{proof}
\begin{figure}[h]
    \centering
    \includegraphics[scale=0.7]{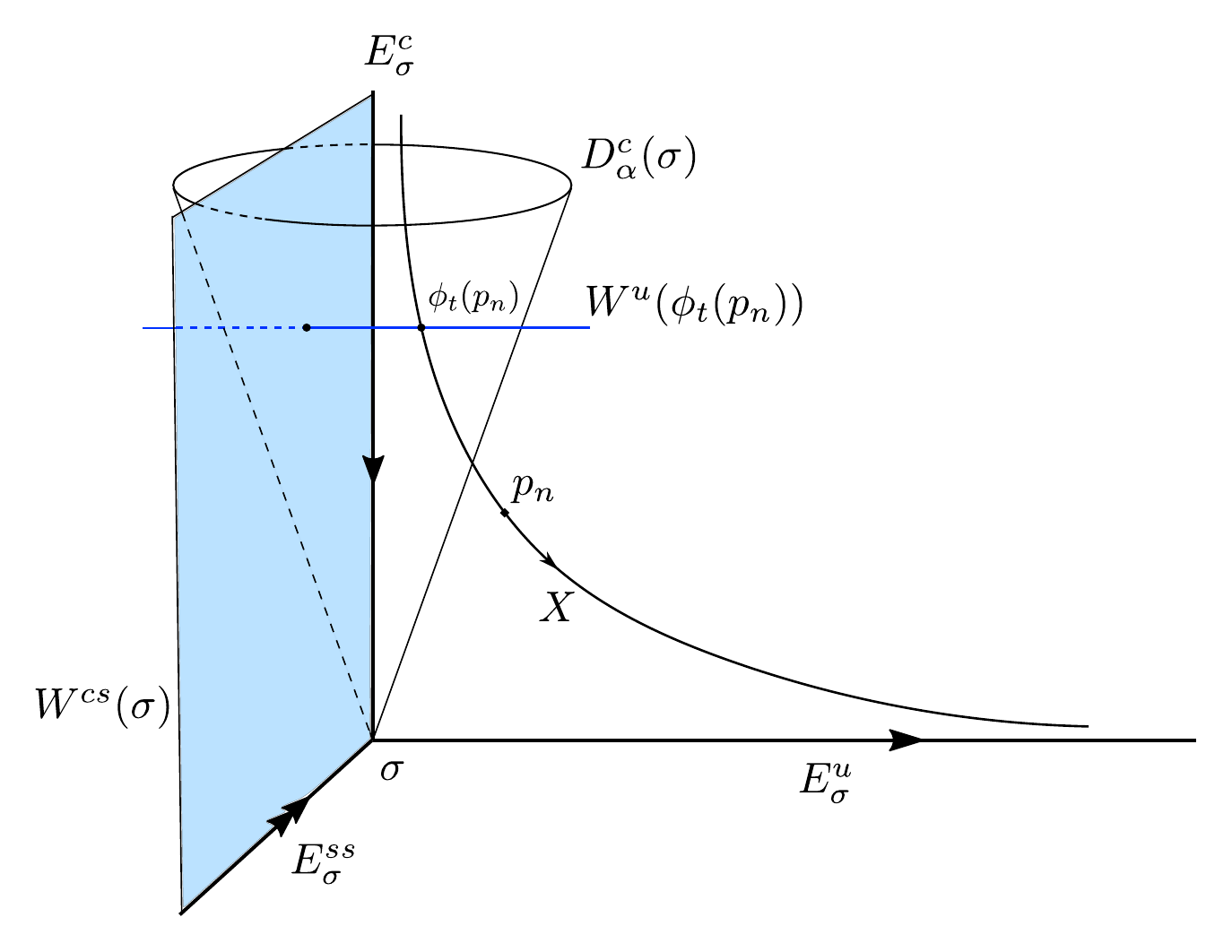}
    \caption{Transverse intersection between $W^u(\Orb(p_n))$ and $W^{cs}(\sigma)$}
    \label{f.2}
\end{figure}
Let $\alpha<\min\{\alpha_1,\alpha_2\}$ where $\alpha_1$ is given by Lemma~\ref{l.u}. 
We claim  that $\{\phi_t(x): t\in H_n\}\cap D^{cs}_{\alpha}(\sigma)\ne\emptyset$. 

To this end, we parse the orbit segment $\phi_{(-t^-_i,t^+_i)}(p_n)$ into three consecutive parts like in the proof of Lemma~\ref{l.timeincone}:
$$
(-t^-_n,t^+_n) = (-t^-_n,-t^s_n)\cup(-t^s_n,t^u_n)\cup(t^u_n,t^+_n),
$$
such that:
\begin{itemize}
\item $-t^s_n$ is the first time in $(-t^-_n,t^+_n)$ such that $\phi_{-t^s_n}(p_n)\notin D^{cs}_{\alpha}(\sigma)$;
\item $t^u_n$ is the first time in $(-t^-_n,t^+_n)$ such that $\phi_{t^u_n}(p_n)\in D^u_{\alpha}(\sigma)$.
\end{itemize}
In other words, the orbit segment in $(t^s_n,t^u_n)$ is `making the turn'. Then according to Lemma~\ref{l.turning.hyperbolic}, $t^u_n+t^s_n$ is uniformly bounded by $T_{\alpha}$. Therefore, we can take $n$ large enough such that 
$$
\nu_n(\{\phi_t(p_n): t\in (t^s_n,t^u_n)\})<\frac{a}{2}.
$$

On the other hand, Lemma~\ref{l.u} states that for every $t\in (t^u_n,t^+_n),$ $\phi_t(p_n)$ cannot be a backward hyperbolic time, since any sub-segment contained in $\phi_{(t^u_n,t)}(p_n)$ cannot be backward contracting. As a result, we have $H_n\cap (t^u_n,t^+_n)=\emptyset$. It then follows that 
$$
\nu_n(\{\phi_t(p_n):t\in H_n\cap(-t^-_n,-t^s_n)\})>\frac a2.
$$
In other words, there is a backward hyperbolic time $y_n = \phi_t(p_n)$ contained in $D^{cs}_{\alpha}(\sigma)$.  

It follows from Lemma~\ref{l.main} that $W^u(y_n)\pitchfork W^{cs}(\sigma)\ne\emptyset$. See Figure~\ref{f.2}. We conclude the proof of the lemma.
\end{proof}

\begin{lemma}[Forward hyperbolic times near $\sigma$]\label{l.cc2}
	Let $\sigma$ be a Lorenz-like singularity. 
	For $\beta>0$ small enough, $\lambda\in(0,1), T_0>0$, there exists $\alpha_3>0$ with the following property: \\
	For every $\alpha<\alpha_3$, let $D$ be a $(\dim E^u+1)$-dimensional disk that contains $\sigma$ and is tangent to $C_\beta(E^{cu})$, and let $z$ be a $(\lambda, T_0)$-forward hyperbolic time that is contained in $D_\alpha^{c}(\sigma)\cap B_r(\sigma)$ for some $r>0$ small enough. Then we have 
	$$
	W^s(z)\pitchfork D\ne\emptyset, \mbox{ for all $n$ large enough}.
	$$
\end{lemma}
\begin{proof}
	Similar to Lemma~\ref{l.main}, there is $\delta>0$ such that $W^s(z)$ has size $\delta \|X(z)\|$, tangent to $C_{\beta'}(E^{ss})$ for some $\beta'>0$ small enough. 

	Following the idea of Lemma~\ref{l.intersection}, we have 
	$$
	\diam W^s(z) > C_1\delta d(z,\sigma) > CC'\delta x^c.
	$$
	On the other hand, since $D$ is tangent to  $C_\beta(E^{cu})$ and satisfies
	$$
	\dim D + \dim W^s(z) = \dim E^u+ 1 + \dim E^{ss} = \dim M,
	$$
	there exists $C'''>0$ such that whenever $\diam W^s(z) > C'' x^{ss}$ we must have $W^s(z)\pitchfork D\ne\emptyset$. See Figure~\ref{pic.cc2}. Since $x^{ss} <\alpha x^c$ inside the center cone $D_\alpha^{c}(\sigma)$, the choice of $\alpha_3 = C_1C'\delta / C'''$ satisfies the requirement of the lemma. 
\begin{figure}[h]
\centering
\includegraphics[scale=0.6]{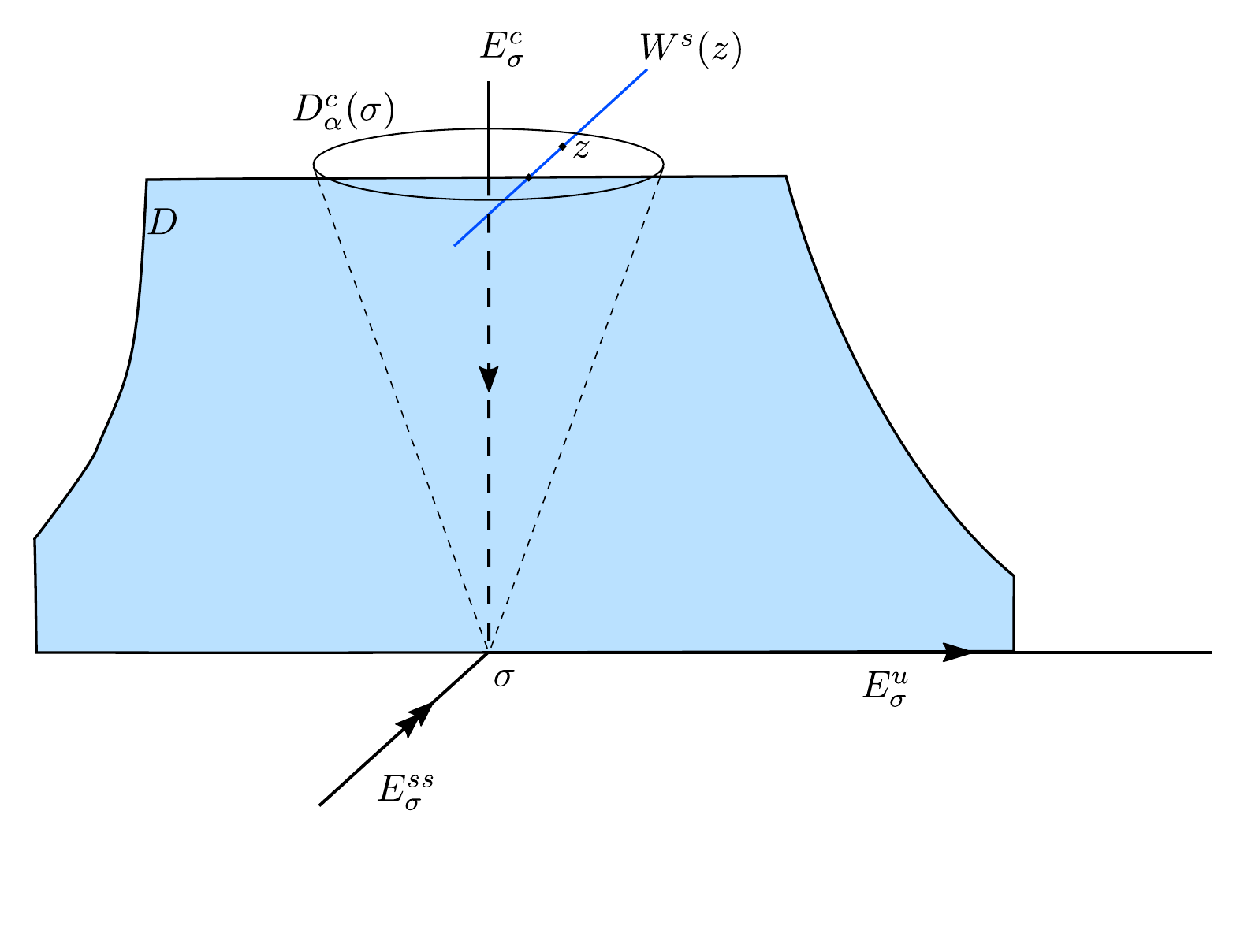}
\caption{Transverse intersection between $W^s(z)$ and $D$}
\label{pic.cc2}
\end{figure}
\end{proof}

Henceforth,  we assume that $r>0$ is small enough such that all the previous lemmas hold. Let $\lambda\in(0,1),T_0>0$ be given by Lemma~\ref{l.u}, and $\alpha<\min\{\alpha_1, \alpha_2,\alpha_3\}$.

\section{Proof of Theorem~\ref{m.star}}\label{ss.starproof}
This section contains the proof of Theorem~\ref{m.star}. The proof is three-fold: 
\begin{enumerate}
\item every chain recurrent class $C$ with positive entropy must contain a periodic orbit;  note that the converse is also true: for $C^1$ generic flows, a non-trivial chain recurrent class containing a periodic point $p$ must coincide with the homoclinic class of $p$ (\cite[Proposition 4.8]{PYY}; see also~\cite{BC}), therefore has positive topological entropy;
\item there exists a neighborhood $U$ of $C$, such that every periodic orbit in $U$ must be contained in $C$;
\item $C$ is isolated.
\end{enumerate}

Recall that $\cR$ is the residual set in $\mathscr{X}_*^1(M)$ described in Section~\ref{s.starflow}. Denote by $\cR_0\subset \mathscr{X}^1(M)$ the residual set of Kupka-Smale vector fields. The next lemma takes care of Step (3) above.

\begin{lemma}\label{l.isolated}
For a $C^1$ vector field $X\in \cR\cap \cR_0$, let $C$ be a chain recurrent class with the following property: there exists a neighborhood $U$ of $C$ such that every periodic orbit in $U$ is contained in $C$.\footnote{Recall that for $C^1$ generic diffeomorphisms, every non-trivial chain recurrent class is approximated by periodic orbits. See~\cite{C06}.} Then $C$ is isolated.
\end{lemma}

\begin{proof}
The proof is standard. Let $C_n$ be a sequence of distinct chain recurrent classes approaching $C$ in the Hausdorff topology. Without loss of generality we may assume that $C_n\subset U$, and $C_n\ne C$ for all $n$. Since $X$ has only finitely many singularities, we may assume that $U$ is small enough such that all the singularities in $U$ are indeed contained in $C$. It then follows that  $C_n\cap\Sing(X)=\emptyset$ for all $n$ large enough. 



Since $X$ is a star vector field, the main result of~\cite{GW} shows that $C_n$ is uniformly hyperbolic. In particular, there exists periodic orbit $\Orb(p_n)\subset C_n\subset U$. By assumption we must have $\Orb(p_n)\subset C$, so $C_n = C$ which is a contradiction. 

\end{proof}

Next, we turn our attention to Step (1) and (2).

\subsection{$C$ contains a periodic orbit}
Step (1) of the proof is carried out in the following two  propositions, which are of independent interest.

\begin{proposition}\label{p.4.2}
There exists a residual set $\tilde\cR\subset \mathscr{X}_*^1(M)$ such that for every $X\in \tilde{\cR}$, let $C$ be a chain recurrent class of $X$ with singularities of different indices, then $C$ contains a periodic point and has positive topological entropy.
\end{proposition}

\begin{proof}
Let $X\in \cR\cap \cR_0$ as before. 
Following the discussion in Section~\ref{s.starflow}, let $\sigma^+\in C$ be a Lorenz-like singularity, and $\sigma^-\in C$ be reverse Lorenz-like. We have 
$$
\Ind(\sigma^+)-1 = \Ind(\sigma^-)=\Ind_C,
$$ 
where $\Ind_C$ is also the stable index of the periodic orbits sufficiently close to $C$. We will construct a periodic point $p$ such that Lemma~\ref{l.main} can be applied at $\sigma^+$ for $X$, and at $\sigma^-$ for $-X$. This shows that $p\in C$, which also implies that $C = H(p)$ where $H(p)$ is the homoclinic class of $p$, therefore has positive topological entropy, by~\cite{BC}.

Given a periodic orbit $\gamma$, we denote by $\Pi(\gamma)$ its primary period. 
By~\cite[Lemma 2.1]{GSW} which is originally due to Liao, there is $\tilde\lambda\in (0,1)$, $T>0$ such that for every periodic orbit $\gamma$ of $X$ with periodic $\Pi(\gamma)$ longer than $T$,
 we have 
$$
\prod_{i=0}^{[\Pi(\gamma)/T]-1} \|\psi_T|_{N^s(\phi_{iT}(x))}\|\le \tilde\lambda^{\Pi(\gamma)},
$$ 
and a similar estimate holds on $N^u$. Here $N^s\oplus N^u$ is the hyperbolic splitting on $N_x$ for the linear Poincar\'e flow $\psi_t$. By the Pliss Lemma~\cite{Pliss}, for every $\lambda\in (\tilde\lambda, 1)$ 
there are $(\lambda,T_0)$-backward hyperbolic times for $N^u$ along the orbit of $\gamma$. Moreover, such points have positive density along the orbit of $\gamma$. 

Fix such $\lambda$ and $r>0$ small enough. For each $n>0$, consider the following property:

\noindent(P(n)): there is a hyperbolic periodic orbit $p_n$, such that (see Figure~\ref{f.4}):
\begin{itemize}
\item the time that $\Orb(p_n)$ spends inside $B_r(\sigma^+)$ is more than $(1/2-1/n)\Pi(p_n)$;
\item the time that $\Orb(p_n)$ spends inside $B_r(\sigma^-)$ is more than $(1/2-1/n)\Pi(p_n)$;
\item the time that $\Orb(p_n)$ spends outside $B_r(\sigma^-)\cup B_r(\sigma^-)$ is less than  $\frac1n\Pi(p_n)$.
\end{itemize} 
Clearly this is an open property in $\mathscr{X}^1(M)$. On the other hand, note that $W^{cs}(\sigma^+)$ and $W^{cu}(\sigma^-)$ must have transverse intersection due to the Kupka-Smale theorem. Since $\sigma^\pm$ are in the same chain recurrent class, using the connecting lemma we can create an intersection between $W^u(\sigma^+)$ and $W^{s}(\sigma^-)$, which in turn creates a loop between $\sigma^+$ and $\sigma^-$. Then standard perturbation technique (see~\cite{GSW} for instance) will allow one to create periodic orbits that satisfy the conditions above. Therefore the following property is generic:

\noindent (P'): there exists periodic orbits $\{p_n\}_n$ that approach both $\sigma^\pm$, such that Property P(n) holds for every $n$.  
\begin{figure}[h]
    \centering
    \includegraphics[scale=0.4]{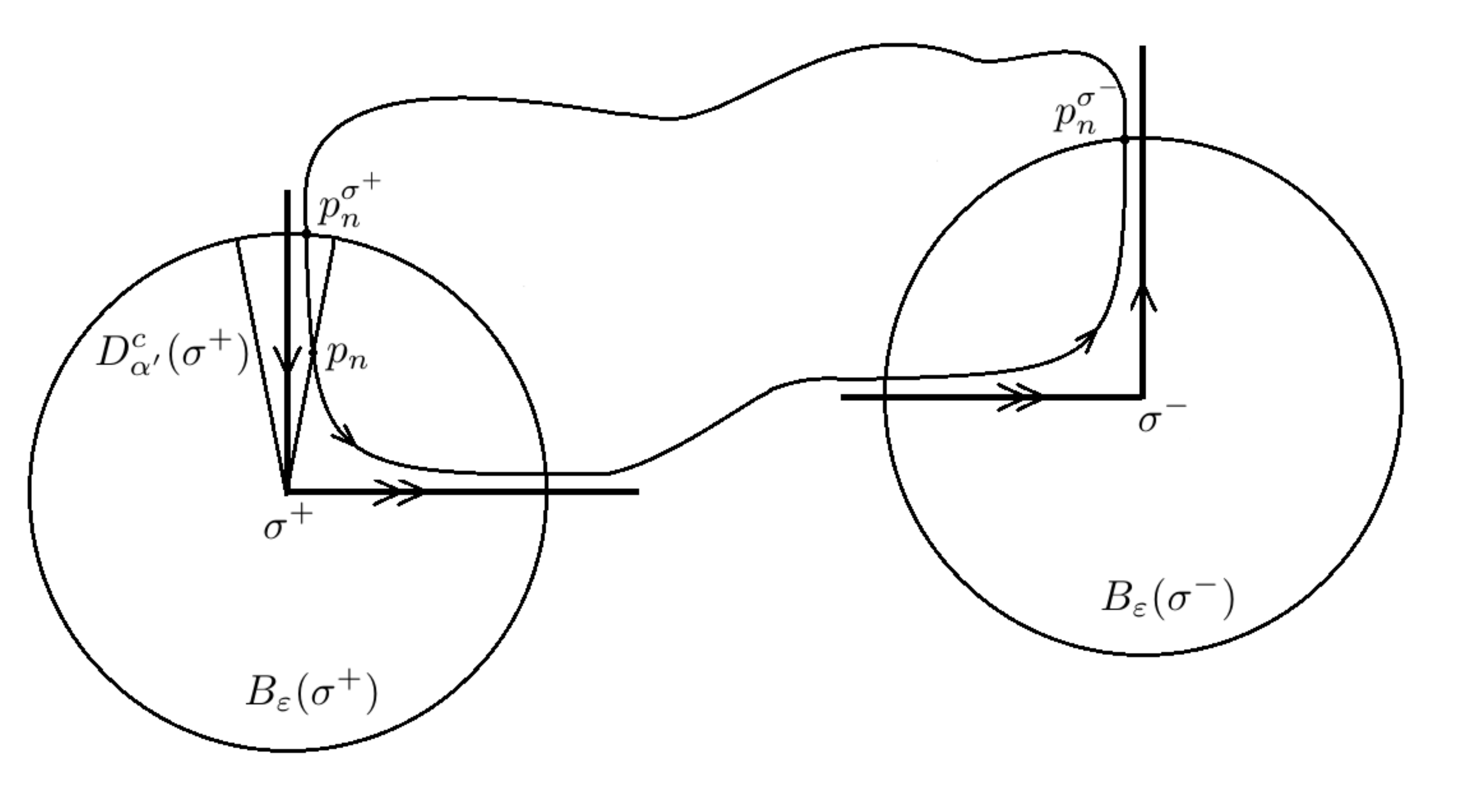}
    \caption{The periodic points $\{p_n\}$ satisfying Property (P')}
    \label{f.4}
\end{figure}

As a result, passing to the generic subset $\cR_1$ where property (P') holds, we may assume that $X$ itself has periodic orbits $\{p_n\}$ satisfying Property (P').

Below we will show that there exists $(\lambda, T_0)$-backward hyperbolic times contained in $\Orb(p_n)\cap B_r(\sigma^+)\cap D_\alpha^{cs}(\sigma^+)$, for $n$ large enough. This allows us to apply Lemma~\ref{l.main} to show that $W^u(\Orb(p_n))$ has transverse intersection with $W^{cs}(\sigma^+)$. The same argument applied to the flow $-X$ shows the transverse intersection between $W^s(\Orb(p_n))$ and $W^{cu}(\sigma^-)$, thus $p_n$ is contained in $C$ for $n$ large enough.

For convenience, we assume that $p_n\to \sigma^+$ such that $p_n$ lies on the boundary of $D^{cs}_{\alpha}(\sigma^+)$, where $\alpha\in(0,\min\{\alpha_1,\alpha_2\})$ as specified at the end of the previous section. Denote by 
$$p_n^{\sigma^+} = \phi_{t_1^n}(p_n)$$ 
where $t_1^n<0$ is the largest real number such that $p_n^{\sigma^+}\in \partial B_r(\sigma^+)$, and 
$$p_n^{\sigma^-} = \phi_{t_2^n}(p_n)$$ 
where 
$
 t^n_2 = \sup\{t<t_1^n: \phi_t(p_n)\in B_r(\sigma^-)\}.
$
See the Figure above. Since the orbit segment $\phi_{[t_2^n, t_1^n]}(p_n)\subset  M\setminus (B_r(\sigma^+)\cup B_r(\sigma^-))$, Property (P') dictates that  $t^n_1-t^n_2<\frac 1n \Pi(p_n)$. 

On the other hand, the Pliss Lemma~\cite{Pliss} shows that the set of $(\lambda,T_0)$-backward hyperbolic times on the orbit of $p_n$ have density $a>0$. Thus for $n$ large enough, there must be hyperbolic times on the subsegment $\Orb(p_n)\cap B_r(\sigma^\pm)$. We consider the following cases:
\begin{enumerate}
\item If  there exists a backward hyperbolic time $p_n'$ contained in the orbit segment $\Orb(p_n)\cap B_r(\sigma^+)$, then by Lemma~\ref{l.u} $p_n'\notin D^{u}_{\alpha}(\sigma^+)$; on the other hand, the transition time from $D^{cs}_{\alpha}(\sigma^+)$ to $D^{u}_{\alpha}(\sigma^+)$ is bounded (Lemma~\ref{l.turning.hyperbolic}); this shows that $p_n'\in D_\alpha^{cs}(\sigma^+)$;
\item 
if there exists a backward hyperbolic times  $p'_n \in \Orb(p_n)\cap B_r(\sigma^-)$, then Lemma~\ref{l.sgood} (applied to $-X$) shows that the orbit segment from $p'_n$ to $p^{\sigma^-}_n$ 
is backward contracted by the scaled linear Poincar\'e flow. By Lemma~\ref{l.glue}, $p_n^{\sigma^-}$ itself must be a backward hyperbolic time.

\noindent Next, by Lemma~\ref{l.u}, the orbit segment from $p^{\sigma^+}_n$ to $p_n$ is backward contracting, and the orbit segment from $p_n^{\sigma^-}$ to $p^{\sigma^+}_n$ has very small length comparing to the former. As a result, $p_n$ is a backward hyperbolic time.

\end{enumerate}

It then follows from both cases, that there exists a backward hyperbolic time inside $\Cl(D_\alpha^{cs}(\sigma^+))$. As a result of Lemma~\ref{l.main}, $W^u(\Orb(p_n))$ must intersect transversally with $W^{cs}(\sigma^+)$.

The same argument applied to the flow $-X$ shows that the stable manifold of $\Orb(p_n)$ intersects transversally with the unstable manifold of $\sigma^-$. We conclude that $\Orb(p_n)$ is contained in the chain recurrent class of $\sigma^\pm$, finishing the proof of this proposition.

\end{proof}

\begin{proposition}\label{p.periodicorbit} There exists a residual set $\tilde\cR\subset \mathscr{X}_*^1(M)$ such that for every $X\in \tilde{\cR}$, if $C$ is a chain recurrent class of $X$ with singularity and positive topological entropy then $C$ contains a periodic point $p$. 
\end{proposition}
\begin{proof}
In view of the discussion in Section~\ref{s.starflow}, we consider the following two cases for $X\in\cR\cap \cR_0$:

\noindent {\em Case 1.} All singularities in $C$ have the same index. Then they must be of the same type: either they are all Lorenz-like or all reverse Lorenz-like. By~\cite[Theorem 3.7]{GSW}, $C$ is sectional hyperbolic for $X$ or $-X$. We may assume that $C$ is sectional hyperbolic for $X$. 
When $h_{top}(\phi_t|_C)>0$, \cite[Theorem B]{PYY} guarantees that there are periodic orbits $p$ contained in $C$. 

\noindent {\em Case 2.} $C$ contains singularities with different indices. Then the proposition follows from Proposition~\ref{p.4.2}.
\end{proof}

This finishes the proof of Step (1). 

\subsection{All periodic orbits close to $C$ must be contained in $C$}

Now we turn our attention to Step (2). The main result of this subsection is Lemma~\ref{l.multisingular} which shows that whenever a periodic orbit $\Orb(p)$ gets sufficiently close to a chain recurrent class containing a periodic orbit, we must have $W^u(p)\pitchfork W^s(x)$ for some $x\in C$. 

The proof of this lemma requires a careful analysis on how the backward hyperbolic times along the orbit of $p$ approach points in $C$. There are essentially three cases:
\begin{enumerate}
\item backward hyperbolic times near a regular point; this is the easy case;
\item backward hyperbolic times near a Lorenz-like singularity $\sigma$; this has been taken careful of by Lemma~\ref{l.cc};
\item backward hyperbolic times near a reverse Lorenz-like singularity $\sigma$; in this case we consider $-X$, under which we have a sequence of forward hyperbolic times near a Lorenz-like singularity $\sigma$; this is taken care of by the following lemma, as well as Lemma~\ref{l.cc2}.
\end{enumerate}

\begin{lemma}\label{l.l}
There exists a residual set $\tilde\cR\subset \mathscr{X}_*^1(M)$ such that for every $X\in \tilde{\cR}$, let $C$ be a chain recurrent class of $X$ that contains some Lorenz-like singularity $\sigma$ and a hyperbolic periodic point $q$. Let $\{p_n\}$ be a sequence of periodic points with $\lim_H\Orb(p_n)\subset C$. Furthermore, assume that there exists $x_n\in\Orb(p_n)$ that are $(\lambda, T_0)$-forward hyperbolic times for $\lambda\in (0,1), T_0>0 $ independent of $n$, with $x_n\to \sigma$.  Then for $n$ large enough,
$$
W^s(\Orb(p_n))\pitchfork W^u(\Orb(q))\ne\emptyset.
$$
\end{lemma}
\begin{proof}
Recall that $E^{ss}_\sigma$ is the subspace spanned by the (generalized) eigenspaces corresponding to the eigenvalues $\lambda_1, \ldots, \lambda_{s-1}$; furthermore we have $E^{ss}_\sigma \oplus E_\sigma^c = E_\sigma^{cs}$.

Consider the strong stable manifold of $\sigma$, $W^{ss}(\sigma)$ that is tangent to $E^{ss}_\sigma$ at $\sigma$. It exists because of the dominated splitting $T_\sigma M = E^{ss}_\sigma \oplus E_\sigma^{cu}$ with  $E_\sigma^{cu} = E^{c}_\sigma\oplus E^u_\sigma$. Moreover, $W^{ss}(\sigma)$ locally divides the stable manifold $W^s(\sigma)$ into two branches, which we denote by $W^{s,\pm}(\sigma)$.

We may assume that ${p_n}$ themselves are forward hyperbolic times with  $p_n\to \sigma$. Fix some $r>0$ small enough, we consider $t_n<0$ the largest real number such that $\phi_{t_n}(p_n)\in \partial B_r(\sigma)$. 
In particular, $z_n:= \phi_{t_n}(p_n) \to z\in W^{s, +}(\sigma)\cap C\cap \partial B_r(\sigma)$. Since the orbit of $p_n$ can only approach $\sigma$ along the center direction $E^c_\sigma$, we must have $z_n\in \Cl(D_\beta^c(\sigma))$, for some $\beta = \beta(r)>0$. By Lemma~\ref{l.sgood}, the orbit segment $\phi_{[t_n,0]}(p_n)$ is forward contracting. It follows from Lemma~\ref{l.glue} that $z_n$ itself is a $(\lambda, T_0)$-forward hyperbolic time, and possesses stable manifold with size proportional to the flow speed at $z_n$. 

On the other hand, we have $W^{s, +}(\sigma)\cap W^u(\Orb(q))\ne\emptyset$, a courtesy of the connecting lemma~\cite{BC} (see also~\cite[Proposition 4.9]{PYY}). Let $a$ be a point of intersection between $W^{s, +}(\sigma)$ and $W^u(\Orb(q))$, and $D$ a disk in $W^u(\Orb(q))$ with $a\in D\subset W^u(\Orb(q))$. By the inclination lemma, $\phi_t(D)$ approximates $W^u(\sigma)$ as $t\to\infty$.

Note that $\tilde D := \{\phi_t(D): t>0\}$ is tangent to the cone $C_\beta(E^{cu})$ and has dimension $\dim E^u_\sigma+ 1$. See figure~\ref{pic.cc2}. Shrinking $r$ such that $\beta(r)<\alpha_3$ and applying Lemma~\ref{l.cc2},\footnote{ $\tilde D$ defined in this way dos not contain $\sigma$. However it can be extended to a disk $\overline D$ that contains $\sigma$ in its interior. Note that $\overline D$ and $\tilde D$ coincide within the upper half of the cone $D^c_\alpha(\sigma)$, which contains the point of the transverse intersection with $W^s(z_n)$.} we obtain $W^s(z_n)\pitchfork \{\phi_t(D): t>0\}\ne\emptyset$ for all $n$ large enough, with which we conclude the proof. 

\end{proof}
\begin{remark}
The proof of the lemma is reminiscent of~\cite[Corollary E]{PYY} with two major differences. Firstly, in~\cite{PYY}, $C$ is Lyapunov stable. This guarantees not only the existence of a periodic orbit $q\in C$, but the disk $D$ is contained in the class $C$, which immediately allows one to conclude that $p_n\in C$ for $n$ large enough. Without Lyapunov stability,  we only obtain the chain attainability from $C$ to $p_n$. Secondly, $C$ is sectional hyperbolic in~\cite{PYY}. Therefore nearby periodic orbits have dominated splitting $E^s\oplus F^{cu}$ and, consequently, have stable manifolds with uniform size.
\end{remark}

Now we are ready to state the main lemma of this subsection, which allows us to establish the chain attainability from nearby periodic orbits to the class $C$. Note that the lemma does not impose any condition on the type of singularities contained in $C$, therefore can be applied to $-X$.

\begin{lemma}\label{l.multisingular}
There exists a residual set $\tilde\cR\subset \mathscr{X}_*^1(M)$ such that for every $X\in \tilde{\cR}$, let $C$ be a chain recurrent class of $X$ with singularities and a hyperbolic periodic point. Then there exists a neighborhood $U$ of $C$ such that every periodic orbit $\Orb(p)\subset U$ satisfies
$$
W^u(p)\pitchfork W^s(x)\ne\emptyset
$$
for some $x\in C$.
\end{lemma}
\begin{proof}
We prove by contradiction. Assume that there is a sequence of periodic orbits $\Orb(p_n)\subset U_n$ with $\cap_n U_n=C$, such that $W^u(p)\pitchfork W^s(x)=\emptyset$ for every $x\in C$. It is easy to see that the period of $p_n$  must tend to infinity; otherwise, we could take a subsequence of $\{p_n\}$ that converges to a periodic orbit $p_0\in C$; by the star assumption, $p_0$ is hyperbolic, and must be homoclinically related with $p_n$ for $n$ large enough, a contradiction.


By~\cite[Theorem 5.7]{GSW}, the index of $p_n$ coincides with $\Ind_C$ for all $n$ large enough. By~\cite[Lemma 2.1]{GSW}, there exists $\lambda\in (0,1), T_0>0$ such that $\Orb(p_n)$ contains $(\lambda,T_0)$-backward hyperbolic times $x_n$. Moreover, the collection of  such points $\Lambda_n = \{x_n\in\Orb(p_n): x_n\mbox{ is a backward  hyperbolic time}\}$ have positive density (independent of $n$) in $\Orb(p_n)$  with respect to the empirical measure on $\Orb(p_n)$, thanks to the Pliss lemma~\cite{Pliss}. 

Taking subsequence if necessary, we write $\tilde{C}\subset C$ for the Hausdorff limit of $\Orb(p_n)$, and $\Lambda = \limsup_{n\to\infty} \Lambda_n\subset \tilde{C}$ (note that it is not invariant). We may also assume that the empirical measure $\mu_n$ on $\Orb(p_n)$ converges to an invariant measure $\mu$ supported on $\tilde{C}$. The backward hyperbolic times having uniform positive density implies that $\mu(\Lambda)>0$. 

We consider the following two cases:

\noindent {\bf Case 1.} There is an ergodic component of $\mu$, denoted by $\mu_1$, which satisfies $\mu_1(\Lambda)>0$ and $\mu_1(\Sing)=0$. 

Then $\mu_1$ must be a non-trivial hyperbolic measure, thanks to~\cite[Theorem 5.6]{GSW}.
The same argument used in Lemma~\ref{l.shadowing} (f) shows that for $n$ large enough, $W^u(\Orb(p_n))$ has transverse intersection with the stable manifold of some point in $\supp \mu_1$. Roughly speaking, every regular point in $\Lambda\cap\,\supp \mu_1$ have stable manifold (with uniform size if we choose a subset of $\Lambda$ away from all singularities; note that such subset still has positive $\mu_1$ measure), which must intersect transversally with the unstable manifold of the backward hyperbolic times $x_n\in\Lambda_n$ (recall that such points have unstable manifold up to the flow speed; if we take them uniformly away from all singularities, then their unstable manifolds also have uniform size), a contradiction. 

\noindent {\bf  Case 2.} Every ergodic component $\mu_1$ of $\mu$ with $\mu_1(\Lambda)>0$ is supported on some singularities $\{\sigma\}$. 

As per our discussion in Section~\ref{s.starflow},   singularities in $C$ must be either Lorenz-like or reverse Lorenz-like. We further consider two subcases:

\noindent {\bf Subcase 1.} One of those $\sigma$ is reverse Lorenz-like\footnote{Note that this case does not happen when $C$ is sectional hyperbolic.}.

In this case we consider the vector field $-X$. Note that backward hyperbolic times for $X$ are forward hyperbolic times for $-X$. Therefore we have a sequence of forward hyperbolic times on the orbit of $p_n$, approaching a Lorenz-like singularity $\sigma$. We are in a position to apply Lemma~\ref{l.l}, which shows that $$
W^{s, -X}(\Orb(p_n))\pitchfork W^{u, -X}(q)\ne\emptyset
$$
for a hyperbolic periodic point $q\in C$. Reverse back to $X$, we see that $W^{u}(\Orb(p_n))\pitchfork W^{s}(x)\ne\emptyset$, which contradictions with our assumption on $C$.

\noindent {\bf Subcase 2.} One of those $\sigma$ is Lorenz-like.

In this case we have backward hyperbolic times in a neighborhood of a Lorenz-like singularity with positive density. By Lemma~\ref{l.cc}, $W^u(\Orb(p_n))$ intersect transversely with $W^{cs}(\sigma^+)$, which is a contradiction.



The proof is now complete.
\end{proof}

In the proof of Lemma~\ref{l.multisingular}, the assumption that $C$ contains a periodic orbit is only used in Case 2, Subcase 1. Since this case is impossible when $C$ is sectional hyperbolic, we obtain the following proposition which has intrinsic interest. 

\begin{proposition}\label{p.l}
Let $C$ be a sectional hyperbolic chain recurrent class for $C^1$ generic star flow $X$. Assume that $C$ contains a singularity $\sigma$. Then there exists a neighborhood $U$ of $C$, such that $C$ is chain attainable from all periodic orbits in $U$.
\end{proposition}

\subsection{Proof of Theorem~\ref{m.star}}

Now we are ready to prove Theorem~\ref{m.star}. By Proposition~\ref{p.periodicorbit}, every singular chain recurrent class with positive entropy contains a periodic point. We then apply Lemma~\ref{l.multisingular} to $C$ for $X$ and $-X$. This shows that there exists a neighborhood $U$ of $C$, such that all periodic orbits in $U$ are homoclinically related with $C$, therefore, are contained in $C$. Finally we use Lemma~\ref{l.isolated} to conclude that $C$ is isolated. 

The only remaining case is chain recurrent classes with zero entropy. By Proposition~\ref{p.4.2}, singularities in such classes must have the same stable index. Then~\cite[Theorem 3.7]{GSW} shows that $C$ is sectional hyperbolic for $X$ or $-X$. $C$ cannot contain any periodic orbit; otherwise it must coincide with the homoclinic class of said periodic orbit, resulting in  positive topological entropy.

To conclude the proof, we need to show that the only ergodic invariant measures supported on $C$ are point masses of singularities. 

For this purpose, let $\mu$ be an ergodic invariant measure with $\supp\mu\subset C$. We prove by contradiction and assume that $\mu\ne\delta_\sigma$ for any $\sigma\in\Sing(X)$. By~\cite{GSW}, $\mu$ is a non-trivial hyperbolic measure whose support contains regular orbits of $X$. By Lemma~\ref{l.C1pesin} and~\ref{l.shadowing} (f), there exists a periodic orbit in $C$, a contradiction.

\subsection{Proof of the corollary}
We conclude this paper with the proof of the Corollary. 

\begin{proof}[Proof of Corollary~\ref{mc.lyapunovstable}]
In~\cite{Liao} it is proven that $C^1$ generic star flows have only finitely many periodic sinks. As a result, if $X$ has infinitely many distinct Lyapunov stable chain recurrent classes $\{C_n\}_{n=1}^\infty$, then we may assume that $C_n$'s are non-trivial and approach, under Hausdorff topology, to a chain recurrent class $C$. Note that $C$ cannot be trivial since trivial chain recurrent classes, i.e., periodic orbits and singularities, are all hyperbolic and  isolated due to the star assumption. Therefore $C$ is non-trivial and sectional hyperbolic with  zero topological entropy due to Theorem~\ref{m.star}.   

Denote by $\lambda_C>0$ the sectional volume expanding rate on $F^{cu}$, 
then by the continuity of the dominated splitting (see~\cite[Appendix B]{BDV} and~\cite{ArPa10}), nearby classes  $C_n$ must be sectional hyperbolic and have sectional volume expanding rate $\lambda_{C_n}>\lambda_C/2$. Since $C_n$ are Lyapunov stable, we are in a position to apply \cite[Theorem C]{PYY} on $C_n$, and see that $h_{top}(X|_{C_n})>\lambda_C/2$. 
By the variational principle, we can take $\mu_n$ an ergodic invariant measure supported on $C_n$, such that $h_{\mu_n}(X)\ge\lambda_C/2$.

On the other hand, \cite[Theorem A]{PYY} shows that $X$ is robustly entropy expansive in s small neighborhood of $C$; in particular, this together with~\cite{B72} shows that the metric entropy must be upper semi-continuous in a small neighborhood of $C$. Let $\mu$ be any limit point of $\mu_n$ in weak-* topology, then we see that $\supp \mu\subset C$ and $h_\mu(X)\ge\lambda_C/2$. It then follows from the variational principle that $h_{top}(X|_C)\ge\lambda_C/2$, a contradiction.

\end{proof}

\end{document}